\documentclass[sn-mathphys,Numbered]{sn-jnl}

\usepackage{graphicx}%
\usepackage{multirow}%
\usepackage{amsmath,amssymb,amsfonts}%
\usepackage{amsthm}%
\usepackage{mathrsfs}%
\usepackage[title]{appendix}%
\usepackage{xcolor}%
\usepackage{textcomp}%
\usepackage{manyfoot}%
\usepackage{booktabs}%
\usepackage{algorithm}%
\usepackage{algorithmicx}%
\usepackage{algpseudocode}%
\usepackage{listings}%

\theoremstyle{thmstyleone}%
\newtheorem{theorem}{Theorem}
\newtheorem{lemma}{Lemma}
\newtheorem{proposition}{Proposition}%

\theoremstyle{thmstyletwo}%
\newtheorem{example}{Example}%
\newtheorem{remark}{Remark}%

\theoremstyle{thmstylethree}%

\begin{document}
	
	\title[Multi-window  Gabor Systems]{Discrete Zak Transform and  Multi-window  Gabor Systems on Discrete Periodic Sets}
	
	
	\author[]{\fnm{Najib} \sur{Khachiaa}}\email{khachiaa.najib@uit.ac.ma}

	\affil[]{\orgdiv{Laboratory Partial Differential Equations, Spectral Algebra and Geometry, Department of Mathematics}, \orgname{Faculty of Sciences, University Ibn Tofail}, \orgaddress{\city{Kenitra}, \country{Morocco}}}

	\abstract{	In this paper, $\mathcal{G}(g,L,M,N)$ denotes a $L-$window Gabor system on a periodic set $\mathbb{S}$, where $L,M,M\in \mathbb{N}$ and $g=\{g_l\}_{l\in \mathbb{N}_L}\subset \ell^2(\mathbb{S})$.  We characterize which $g$ generates  a complete multi-window  Gabor system and a multi-window Gabor frame $\mathcal{G}(g,L,M,N)$ on  $\mathbb{S}$ using the Zak transform.   Admissibility conditions for a periodic set to admit a  complete multi--window  Gabor system, multi-window  Gabor (Parseval) frame, and multi--window  Gabor (orthonormal) basis $\mathcal{G}(g,L,M,N)$ are given with respect to the parameters $L$, $M$ and $N$.}


	
	\maketitle

\section{Introduction and preliminaries}

When a signal appears periodically but intermittently, it can be considered within the entire space $\ell^2(\mathbb{Z})$ and analyzed in the standard manner. However, if the signal is only emitted for short periods, this method might not be the best approach. To perform Gabor analysis of the signal most efficiently while preserving all its features, Li and Lian studied single window Gabor systems on discrete periodic sets. They derived density results and frame characterizations. Compared to single window Gabor systems, multiwindow Gabor systems can be both interesting and beneficial, as they allow for more flexibility by using windows of different types and widths. For certain parameters $N$ and $M$, there does not exist an associated Gabor frame with a single window. However, allowing the use of multiple windows guarantees the existence of Gabor frames. For example, for $\mathbb{S}=\mathbb{Z}$, a Gabor frame with one window exists only if $N\leq M$. It has been shown in \cite{3} that when it is not the case,  by allowing the use of multiple windows, the existence of Gabor frames associated with $L$-windows is ensured, where $L$ is an integer satisfying $N\leq LM$ (which, of course, exists). Among the objectives of this paper is to extend this result to an arbitrary $N\mathbb{Z}$-periodic set $\mathbb{S}$ in $\mathbb{Z}$.

A sequence $\{f_i\}_{i\in \mathcal{I}}$, where $\mathcal{I}$ is a countable set, in a separable Hilbert space $H$ is said to be frame if there exist $0< A\leq B<\infty$ ( called frame bounds) such that for all $f\in H$, $$A\|f\|^2\leq \displaystyle{\sum_{i\in \mathcal{I}}\vert \langle f,f_i \rangle\vert^2}\leq B\|f\|^2.$$ If only the upper inequality holds,  $\{f_i\}_{i\in \mathcal{I}}$ is called  a Bessel sequence with Bessel bound $B$. If $A=B$, the sequence is called a tight frame and if $A=B=1$, it is called a Parseval frame for $H$. For more details on frame theory, the reader can refer to \cite{1}.

Denote by $\mathbb{N}$ the set of positive integers, i.e. $\mathbb{N}:=\{1,2,3,...\}$ and for a given $K\in \mathbb{N}$, write $\mathbb{N}_{K}:=\{0,1,...,K-1\}$. Let $N,M,L\in \mathbb{N}$ and $p,q\in \mathbb{N}$ such that pgcd$(p,q)=1$ and $\displaystyle{\frac{N}{M}=\frac{p}{q}}$. A nonempty subset $\mathbb{S}$ of $\mathbb{Z}$ is said to be $N\mathbb{Z}$-periodic set if for all $j\in \mathbb{S}$ and for all $n\in \mathbb{Z}$, $j+nN \in \mathbb{S}$. For $K\in \mathbb{N}$, write $\mathbb{S}_K:=\mathbb{S}\cap \mathbb{N}_K$. We denote by $\ell^2(\mathbb{S})$ the closed subspace of $\ell^2(\mathbb{Z})$ defined by, $$\ell^2(\mathbb{S}):=\{f\in \ell^2(\mathbb{Z}):\, f(j)=0 \text{ if } j\notin  \mathbb{S} \}.$$ 
Define the modulation operator $E_{\frac{m}{M}}$ with $m\in \mathbb{Z}$ and the translation operator $T_{nN}$ with $n\in \mathbb{Z}$ for $f\in \ell^2(\mathbb{S})$ by: $$E_{\frac{m}{M}}f(.):=e^{2\pi i \frac{m}{M}.} f(.), \;\; T_{nN}f(.):=f(.-nN).$$ 
The modulation and translation  operators are unitary operators of $\ell^2(\mathbb{S})$. For $g:=\{g_l\}_{l\in \mathbb{N}_L} \subset \ell^2(\mathbb{S})$, the associated multiwindow discrete Gabor system (M-D-G) is given by,
$$\mathcal{G}(g,L,M,N):=\{E_{\frac{m}{M}}T_{nN}g_l\}_{m\in \mathbb{N}_M,n\in \mathbb{Z}, l\in \mathbb{N}_L}.$$ 
For $j\in \mathbb{Z}$, we denote $\mathcal{K}_j=\{k\in \mathbb{N}_p:\, j+kM\in \mathbb{S}\}$ and  $\mathcal{K}(j):=diag(\chi_{\mathcal{K}_j}(0),\chi_{\mathcal{K}_j}(1),...,\chi_{\mathcal{K}_j}(p-1))$.

Let $K\in \mathbb{N}$.  The discrete Zak tansform $z_K$ of $f\in \ell^2(\mathbb{Z})$ for $j\in \mathbb{Z}$ and a.e $\theta\in \mathbb{R}$ is defined by, 	
$$z_Kf(j,\theta):=\displaystyle{\sum_{k\in \mathbb{Z}}f(j+kK)e^{2\pi i k \theta}}.$$
$z_Kf$ is quasi-periodic. i.e. $\forall j, k,l\in \mathbb{Z},\,\theta\in \mathbb{R}$ we have: $$z_Kf(j+kK,\theta+l)=e^{-2\pi ik\theta}z_Kf(j,\theta).$$
Then $z_K$ is, completely, defined by its values for $j\in \mathbb{N}_K$ and $\theta \in [0,1[$.

This paper is organized as follows. In section 2,  we will present some auxiliary lemmas to be used in the following sections. In section 3, we characterize which $g\in \ell^2(\mathbb{S})$ generates a complete multi-window discrete Gabor system and a multi-window discrete Gabor frame $\mathcal{G}(g,L,M,N)$ for $\ell^2(\mathbb{S})$ using the Zak transform. In section 4, we provide an admissibility characterization for complete multi-window discrete Gabor systems and multi-window discrete Gabor frames $\mathcal{G}(g,L,M,N)$ 
on a discrete periodic set $\mathbb{S}$, and we finish with an example.

\section{Auxiliary lemmas}
In this section, we present several lemmas and introduce the notations that will be utilized in the following sections. In addition to the notations introduced in the introduction, let $\mathcal{M}_{s,t}$ denote the set of all $s\times t$ matrices with entries in $\mathbb{C}$. We use $p\wedge q$ to indicate that $p$ and $q$ are coprime. For a given matrix $A$, $A^*$ is the conjugate transpose of $A$, $N(A)$ represents its kernel, and $A_{s,t}$ refers to its $(s,t)$-component. When $A$ is a column matrix, we denote its $r$-component simply by $A_r$. Following this, we provide several definitions and results that will be useful throughout the rest of the paper.
\begin{lemma}\cite{4}\label{lem1.1}
	Let $K\in \mathbb{N}$, and $\mathbb{S}$ be a $K\mathbb{Z}$-periodic set in $\mathbb{Z}$. Write $\mathbb{S}_K=\mathbb{S}\cap \mathbb{N}_K$. Then the restriction of $z_K$ to $\ell^2(\mathbb{S})$ is a unitary linear operator from $\ell^2(\mathbb{S})$ to the Hilbert space  $\ell^2(Q)$ where $Q=\mathbb{S}_K\times [0,1[$ and 
	$$\ell^2(Q):=\{\psi:Q \rightarrow \mathbb{C}:\, \displaystyle{\sum_{j\in \mathbb{S}_K}\int_{0}^1 \vert \psi(j,\theta)\vert^2 \, d\theta}< \infty\}.$$ 
\end{lemma} 

Let $A,B\subset \mathbb{Z}$ and $K\in \mathbb{N}$.  We say that $A$ is $K\mathbb{Z}$-congruent to $B$ if there exists a partition $\{A_k\}_{k\in \mathbb{Z} }$ of $A$ such that $\{A_k+kK\}_{k\in \mathbb{Z}}$ is a partition of $B$.
\begin{lemma}\cite{4}
	Let $N,M\in \mathbb{N}$ and $p,q\in \mathbb{N}$ such that $p\wedge q=1$ and $\displaystyle{\frac{N}{M}=\frac{p}{q}}.$ Then 
	the set \\$\Delta:=\{j+kM-rN:\; j\in \mathbb{N}_{\frac{M}{q}},\, k\in \mathbb{N}_p,\, r\in \mathbb{N}_q\}$ is $qN$-congruent to $\mathbb{N}_{pM}$.
\end{lemma}
For each $f\in \ell^2(\mathbb{Z})$, we associate a matrix-valued function $Z_f:\mathbb{Z}\times \mathbb{R}\rightarrow \mathcal{M}_{q,p}$ whose entry at the r-th row and the k-th column is defined by
$$Z_f(j,\theta)_{r,k}=z_{pM}f(j+kM-rN,\theta).$$
\begin{lemma}\cite{4}
	Let $N,M\in \mathbb{N}$ and $p,q\in \mathbb{N}$ such that $p\wedge q=1$ and $\displaystyle{\frac{N}{M}=\frac{p}{q}}.$ Then 
	$z_{pM}f$ is completeley determined by the matrices $Z_f(j,\theta)$ for $j\in \mathbb{N}_{\frac{M}{q}}$ and $\theta \in [0,1[$.\\
	Conversely, a matrix-valued function $Z:\mathbb{N}_{\frac{M}{q}}\times [0,1[\rightarrow \mathcal{M}_{q,p}$ such that for all $j\in \mathbb{N}_{\frac{M}{q}}$, $Z(j,.)_{r,k}\in L^2([0,1[)$ also determines a unique $f\in \ell^2(\mathbb{Z})$ such that for all $j\in \mathbb{N}_{\frac{M}{q}}$, $\theta \in [0,1[$, $Z_f(j,\theta)=Z(j,\theta).$
	
	For $g:=\{g_l\}_{l\in \mathbb{N}_L}\in \ell^2(\mathbb{Z})$, we associate the matrix-valued function\\ $Z_g:\mathbb{N}_{\frac{M}{q}}\times \mathbb{R}\rightarrow \mathcal{M}_{qL,p}$ defined for all $j\in \mathbb{N}_{\frac{M}{q}},\, \theta \in \mathbb{R}$ by the block matrix: $$Z_g(j,\theta)=\begin{pmatrix}
		Z_{g_0}(j,\theta)\\
		Z_{g_1}(j,\theta)\\
		\vdots\\
		Z_{g_{L-1}}(j,\theta)
	\end{pmatrix}.
	$$
\end{lemma}
\begin{lemma}\cite{4}\label{lem1.4}
	Let $p,q\in \mathbb{N}$ such that $p\wedge q=1$. Then for all $j\in \mathbb{Z}$, there exists a unique $(k_0,l_0)\in \mathbb{N}_p\times \mathbb{Z}$ and a unique $(k_0,m_0,r_0)\in \mathbb{N}_p\times \mathbb{Z}\times \mathbb{N}_q$ such that $j= k_0q+l_0p =k_0q+(m_0q+r_0)p.$
\end{lemma}
\begin{lemma}\cite{4}\label{lem2.5}
	Let $M, N \in \mathbb{N}$ and $p, q \in \mathbb{N}$ such that  $\displaystyle{\frac{N}{M}=\frac{p}{q}}$ and $p\wedge q = 1$. Then, for
	all $m \in \mathbb{Z}$, there exists a unique $(j, r, k, \ell) \in \mathbb{N}_{\frac{M}{q}} \times \mathbb{N}_q \times \mathbb{N}_p \times \mathbb{Z}$ such that $m =
	j+kM -rN +\ell qN$.
\end{lemma}

The following proposition characterizes which Multi-window Gabor frames are Multi-window Gabor Riesz beses using the parameters $L,M$ and $N$.
\begin{proposition}\cite{3} \label{KRK}\hspace{0.3cm}
	Let $g:=\{g_l\}_{l\in \mathbb{N}_L}\subset \ell^2(\mathbb{S})$.
	\begin{enumerate}
		\item  $\mathcal{G}(g,L,M,N)$ is a frame for $\ell^2(\mathbb{S})$ only when $\displaystyle{\frac{card(\mathbb{S}_N)}{M}}\leq L$.
		\item  Assume that $\mathcal{G}(g,L,M,N)$ is a frame for $\ell^2(\mathbb{S})$. Then following statements are equivalent: 
		\begin{enumerate}
			\item $\mathcal{G}(g,L,M,N)$ is a  Riesz basis  (exact frame) for $\ell^2(\mathbb{S})$.
			\item $\displaystyle{\frac{card(\mathbb{S}_N)}{M}}=L$.
		\end{enumerate}
	\end{enumerate}
\end{proposition}

\begin{lemma}\cite{4}\label{lem2.6}
	Let $M \in \mathbb{N}$ and $E \subset \mathbb{Z}$. Then the following conditions are equivalent:
	\begin{enumerate}
		\item $\left\{ e^{2\pi im/M} \cdot \chi_E(\cdot) : m \in \mathbb{N}_M \right\}$ is a tight frame for $\ell^2(E)$ with frame bound $M$.
		\item $\left\{ e^{2\pi im/M} \cdot \chi_E(\cdot) : m \in \mathbb{N}_M \right\}$ is complete in $\ell^2(E)$.
		\item $E$ is $M\mathbb{Z}$-congruent to a subset of $\mathbb{N}_M$.
		\item $\displaystyle{\sum_{k \in \mathbb{Z}} \chi_E(\cdot + kM) \leq 1}$ on $\mathbb{Z}$.
	\end{enumerate}
\end{lemma} 
\begin{lemma}\cite{1}\label{PRO}
	Let $\{f_i\}_{i\in \mathcal{I}}$, where $\mathcal{I}$ is a countable sequence,  be a Parseval frame for a separable Hilbert space $H$. Then the following statements are equivalent:
	\begin{enumerate}
		\item $\{f_i\}_{i\in \mathcal{I}}$ is a Riesz basis.
		\item $\{f_i\}_{i\in \mathcal{I}}$ is an orthonormal basis.
		\item For all $i\in \mathcal{I}$, $\|f_i\|=1.$
	\end{enumerate}
\end{lemma}
\begin{lemma}\label{lem2.8}
	Let $f\in \ell^2(\mathbb{Z})$. If $f\in \ell^2(\mathbb{S})$, then for all $j\in \mathbb{Z}$, a.e $\theta\in \mathbb{R}$, $$Z_f(j,\theta)\mathcal{K}(j)=Z_f(j,\theta).$$
	
\end{lemma}
\begin{proof}
	Let $s\in \mathbb{N}_q$ and $t\in \mathbb{N}_p$. We have: 
	$$\begin{array}{rcl}
		\left(Z_f(j,\theta)\mathcal{K}(j)\right)_{s,t}&=&\displaystyle{\sum_{k=0}^{p-1}Z_f(j,\theta)_{s,k}\mathcal{K}(j)_{k,t}}\\
		&=&\displaystyle{\sum_{k=0}^{p-1}Z_f(j,\theta)_{s,k}\delta_{k,t}\chi_{\mathcal{K}_j}(t)}\\
		&=&Z_f(j,\theta)_{s,t}\chi_{\mathcal{K}_j}(t)\\
		&=&\left\lbrace
		\begin{array}{rcl}
			&Z_f(j,\theta)_{s,t}& \; \text{ if } t\in \mathcal{K}_j,\\
			&0& \;\; \text{ otherwise}. 
		\end{array}
		\right.
	\end{array}$$
	On the other hand, we have $Z_f(j,\theta)_{s,t}=z_{pM}f(j+tM-sN,\theta)=\displaystyle{\sum_{k\in \mathbb{Z}}f(j+tM-sN+kpM)e^{2\pi i k\theta)}}=\displaystyle{\sum_{k\in \mathbb{Z}}f(j+tM-sN+kqN)e^{2\pi i k\theta}}$ since $pM=qN$. Then, if $t\notin \mathcal{K}_j$, $j+tM\notin \mathbb{S}$, then, for all $k\in \mathbb{Z}$, $j+tM-sN+kqN\notin \mathbb{S}$ by the $N\mathbb{Z}$-periodicity of $\mathbb{S}$, thus $f(j+tM-sN+kqN)=0$ for all $k\in \mathbb{Z}$. Hence $Z_f(j,\theta)_{s,t}=0$ if $t\notin \mathcal{K}_j$. The proof is completed.
\end{proof}

\begin{lemma}\label{lem2.9}
	For all $j\in \mathbb{Z}$, $\mathcal{K}(j)$ is an orthogonal projection on $\mathbb{C}^p$. i.e.
	\begin{enumerate}
		\item $\mathcal{K}(j)^2=\mathcal{K}(j)$.
		\item  $\mathcal{K}(j)^*=\mathcal{K}(j)$.
	\end{enumerate}
\end{lemma}
\begin{proof}
	Let $j\in\mathbb{Z}$. We have:  $$\begin{array}{rcl}
		\mathcal{K}(j)^2&=&\text{diag}(\chi_{\mathcal{K}_j}(0)^2,\chi_{\mathcal{K}_j}(1)^2,\ldots,\chi_{\mathcal{K}_j}(p-1)^2)\\
		&=&\text{diag}(\chi_{\mathcal{K}_j}(0),\chi_{\mathcal{K}_j}(1),\ldots, \chi_{\mathcal{K}_j}(p-1))=\mathcal{K}(j).
	\end{array}$$
	And $$\begin{array}{rcl}
		\mathcal{K}(j)^*&=&\text{diag}(\overline{\chi_{\mathcal{K}_j}(0)},\overline{\chi_{\mathcal{K}_j}(1)},\ldots,\overline{\chi_{\mathcal{K}_j}(p-1)})\\
		&=&\text{diag}(\chi_{\mathcal{K}_j}(0),\chi_{\mathcal{K}_j}(1),\ldots, \chi_{\mathcal{K}_j}(p-1))=\mathcal{K}(j).
	\end{array}$$
\end{proof}
\section{Characterizations of complete  multiwindow  discrete Gabor systems and multiwindow discrete Gabor frames}
In this section we use all the notations already introduced Without introducing them again. Let $L,M,N\in \mathbb{N}$ and $p,q\in \mathbb{N}$ such that $p\wedge q=1$ and  $\displaystyle{\frac{N}{M}=\frac{p}{q}}$ and denote $\mathbb{S}_N=\mathbb{S}\cap \mathbb{N}_N$. We characterise what $g=:\{g_l\}_{l\in \mathbb{N}_L}\subset \ell^2(\mathbb{S})$ generates  a complete  Gabor system and  a Gabor frame $\mathcal{G}(g,L,M,N)$.
We first present the following proposition:
\begin{proposition}\label{prop2.1}
	Let $g:=\{g_l\}_{l\in \mathbb{N}_L}\subset \ell^2(\mathbb{S})$. Let $M,N\in \mathbb{N}$ and $p,q\in \mathbb{N}$ such that $p\wedge q=1$. Then the integer-valued function $(j,\theta)\rightarrow rank(Z_g(j,\theta))$ is $\displaystyle{\frac{M}{q}}$-periodic with respect to $j$ and $1$-periodic with respect to $\theta$. Moreover,  for all $j\in \mathbb{Z}$ and a.e $ \theta \in \mathbb{R}$, we have: 
	\begin{equation}
		rank(Z_g(j,\theta))\leq \; card(\mathcal{K}_j).
	\end{equation}
\end{proposition}
For the proof, we need the following lemma:
\begin{lemma}\label{lem2.2}
	For all $j\in \mathbb{Z}$,  a.e $\theta \in \mathbb{R}$, $k'\in \mathbb{N}_p$ and $r'\in \mathbb{N}_q$, we have: 
	$$rank(Z_g(j,\theta))=rank(Z_g(j+k'M+r'N,\theta)).$$
\end{lemma}
\begin{proof}
	Let $j\in \mathbb{Z}$, $\theta \in \mathbb{R}$. Denote by $C_k(j,\theta)$ the $k-$th column of $Z_g(j,\theta)$.\\
	We have for all $l\in \mathbb{N}_L$ and for all $r\in \mathbb{N}_q$, $z_{pM}g_l((j+k'M)+kM-rN,\theta)=\\z_{pM}g_l(j+(k+k_0)M-rN,\theta)$.
	If $0\leq k\leq p-k'-1$, then $C_k(j+k_0M,\theta)=C_{k+k_0}(j,\theta)$. Otherwise  ($p-k'\leq k\leq p-1$), we have by quasi-periodicity of the Zak transform $z_{pM}g_l((j+k'M)+kM-rN,\theta)=e^{-2\pi i \theta}.z_{pM}g_l(j+(k+k'-p)M-rN,\theta).$ Then $C_k(j+k'M,\theta)=e^{-2\pi i\theta}.C_{k+k'-p}(j,\theta).$\\
	Consider the map: $$\phi:\begin{array}{rcl}
		\mathbb{N}_p&\rightarrow& \mathbb{N}_p\\
		k&\mapsto& \left\lbrace\begin{array}{rcl}
			k+k' \;\;\text{ if }  0\leq k\leq p-k'-1, \\
			k+k'-p \;\;\text{ if } p-k'\leq k\leq p-1.
		\end{array}
		\right.
	\end{array}$$
	We show, easily, that $\phi$ is injective. It is, then, bijective. Hence: 
	$$\begin{array}{rcl}
		rank(Z_g(j,\theta))&=&rank\{C_k(j,\theta):\; k\in \mathbb{N}_p\}\\
		&=&rank\{C_k(j+k'M,\theta):\; k\in \mathbb{N}_p\}\\
		&=&rank(Z_g(j+k'M,\theta)).
	\end{array}$$
	Denote by $R_r(j,\theta)$ the $r$-th row of $Z_g(j,\theta)$. Then there exists a unique $(l,r_0)\in \mathbb{N}_L\times \mathbb{N}_q$ such that $r=lq+r_0$. Then ($\forall j_0\in \mathbb{Z}$) $R_r(j_0,\theta)$ is the $r_0$-th row of $Z_{g_l}(j_0,\theta)$.
	We have $z_{pM}g_l((j+r'N)+kM-r_0N,\theta)=z_{pM}g_l(j+kM-(r_0-r')N,\theta)$.\\
	If $r'\leq r_0\leq q-1$, then $R_r(j+r'N,\theta)$ is the $(r_0-r')-$th row of $Z_{g_l}(j,\theta)$, thus $R_r(j+r'N,\theta)=R_{r-r'}(j,\theta)$. Otherwise ($0\leq r_0\leq r'-1$), since $pM=qN$ and by quasi-periodicity of the Zak transform, we have: $z_{pM}g_l((j+r'N)+kM-r_0N,\theta)=z_{pM}g_l(j+kM-(r_0-r'+q)N,\theta)$. Then $R_r(j+r'N,\theta)$ is the $(r_0-r'+q)-$th row of $Z_{g_l}(j,\theta)$, thus $R_r(j+r'N,\theta)=R_{r-r'+q}(j,\theta)$.\\
	Consider the map: $$\psi:\begin{array}{rcl}
		\mathbb{N}_{qL}&\rightarrow& \mathbb{N}_{qL}\\
		r&\mapsto& \left\lbrace\begin{array}{rcl}
			r-r' \;\;\text{ if }  lq+r'\leq r\leq (l+1)q-1, \\
			r-r'+q \;\;\text{ if } lq\leq r\leq lq+r'-1.
		\end{array}
		\right.
	\end{array}$$
	It is easy to show that $\psi$ is injective. Then it is bijective. Hence: 
	$$\begin{array}{rcl}
		rank(Z_g(j,\theta))&=&rank\{R_r(j,\theta):\; r\in \mathbb{N}_{qL}\}\\
		&=&rank\{R_r(j+r'N,\theta):\; r\in \mathbb{N}_{qL}\}\\
		&=&rank(Z_g(j+r'N,\theta)).
	\end{array}$$
	Hence For all $j\in \mathbb{Z}$,  $\theta \in \mathbb{R}$, $k'\in \mathbb{N}_p$ and $r'\in \mathbb{N}_q$, we have: 
	$$rank(Z_g(j,\theta))=rank(Z_g(j+k'M+r'N,\theta)).$$
\end{proof}
\begin{proof}[Proof of proposition \ref{prop2.1}]
	Let $j\in \mathbb{Z}, \, \theta\in \mathbb{R}$. Given an arbitrary $s\in \mathbb{Z}$. By lemma \ref{lem1.4}, there exists a unique $(k_0,m_0,r_0)\in \mathbb{N}_p\times \mathbb{Z}\times \mathbb{N}_q$ such that $s=k_0q+(m_0q+r_0)p$. Then: 
	$$\begin{array}{rcl}
		Z_g(j+\displaystyle{\frac{M}{q}}s,\theta)&=&Z_g(j+k_0M+m_0pM+r_0N,\theta)\\
		&=&e^{2\pi i m_{0}\theta}.Z_g(j+k_0M+r_0N,\theta)\\
		&=&e^{2\pi i m_{0}\theta}.Z_g(j,\theta) \;\;\; \text{ lemma }\ref{lem2.2}.
	\end{array}$$
	Hence: $$rank\left(Z_g(j+\displaystyle{\frac{M}{q}}s,\theta)\right)=rank(Z_g(j,\theta)).$$
	The $1$-periodicity with respect to $\theta$ is simply due to the $1$-periodicity of the Zak transform with respect to $\theta$. 
	On the other hand, if $k\notin \mathcal{K}_j$, i.e. $k$ is such that $j+kM\notin \mathbb{S}$, then, by $N\mathbb{Z}-$periodicity of $\mathbb{S}$, for all $r\in \mathbb{Z}$, $j+kM-rN\notin \mathbb{S}$, and thus for all $l\in \mathbb{N}_L$ and for all $r\in \mathbb{N}_q$, $z_{pM}g_l(j+kM-rN,\theta)=0$ since $pM=qN$. Then the $k$-th column of $Z_g(j,\theta)$ is identically zero. Hence $rank(Z_g(j,\theta))\leq card(\mathcal{K}_j)$.
\end{proof}
\begin{remark}\label{rem1}
	Let $j\in \mathbb{Z}$. Since $\{\mathbb{S}_N+nN\}_{n\in \mathbb{Z}}$ is a partition of $\mathbb{S}$, we have: 
	$$\begin{array}{rcl}
		card(\mathcal{K}_j)&=&\displaystyle{\sum_{k=0}^{p-1}\chi_{\mathbb{S}}(j+kM)}\\
		&=&\displaystyle{\sum_{k=0}^{p-1}\sum_{n\in \mathbb{Z}}\chi_{\mathbb{S}_N}(j+nN+kM)}\\
		&=&\displaystyle{\sum_{k=0}^{p-1}\sum_{n\in \mathbb{Z}}\chi_{\mathbb{S}_N}(j+\displaystyle{\frac{np+kq}{q}M})\;\;\; \text{ since }N=\displaystyle{\frac{pM}{q}}}\\&=&\displaystyle{\sum_{n\in \mathbb{Z}}\chi_{\mathbb{S}_N}(j+\displaystyle{\frac{M}{q}n})\;\;\; \text{ lemma \ref{lem1.4}}}.
	\end{array}$$
	Hence $card(\mathcal{K}_j)$ is $\displaystyle{\frac{M}{q}}$-periodic. Then  the inequality (1) holds for all $j\in \mathbb{Z}$ and a.e $\theta \in \mathbb{R}$ if and only if it holds for all $j\in \mathbb{N}_{\frac{M}{q}}$ and a.e $\theta\in [0,1[$.\\
\end{remark}

The following lemma is very useful for the rest.
\begin{lemma}\label{lem2.4}
	Let $g:=\{g_l\}_{l\in \mathbb{N}_L}\subset \ell^2(\mathbb{Z})$. Let $f\in \ell^2(\mathbb{Z})$. Then the following statements are equivalent:
	\begin{enumerate}
		\item $f$ is orthogonal to $\mathcal{G}(g,L,M,N)$.
		\item For all $j\in \mathbb{N}_M$, a.e. $\theta\in [0,1[$, $Z_g(j,\theta)F(j,\theta)=0$.
	\end{enumerate}
	Where $F(j,\theta):=(\overline{z_{pM}f(j+kM,\theta)})_{k\in \mathbb{N}_p}^t$ for $j\in \mathbb{N}_M
	$ and a.e $\theta\in [0,1[$.
\end{lemma}
\begin{proof}
	We have: \\
	$f$ is orthogonal to $\mathcal{G}(g,L,M,N)$ $\Longleftrightarrow$ $f$ is orthogonal to $\mathcal{G}(g_l,M,N)$ for all $l\in \mathbb{N}_L$.\\
	And we have:\\ 
	$Z_g(j,\theta)F(j,\theta)=0 $ $\Longleftrightarrow$ $Z_{g_l}(j,\theta)F(j,\theta)=0$ for all $l\in \mathbb{N}_L$.\\
	These equivalences together with lemma 3.1 in \cite{4} complete the proof.
\end{proof}

The following proposition characterizes complete multi-window Gabor systems on $\mathbb{S}$.
\begin{theorem}\label{prop3.5}
	Let $g:=\{g_l\}_{l\in \mathbb{N}_L}\subset \ell^2(\mathbb{S})$. Then the following statements are equivalent:
	\begin{enumerate}
		\item $\mathcal{G}(g,L,M,N)$ is complete in $\ell^2(\mathbb{S})$.
		\item For all $j\in \mathbb{N}_{\frac{M}{q}}$, a.e $\theta \in [0,1[$, $rank(Z_g(j,\theta))=card(\mathcal{K}_j)$.
		\item For all $j\in \mathbb{Z}$, a.e $\theta\in \mathbb{R}$, $rank(Z_g(j,\theta))=card(\mathcal{K}_j)$.
	\end{enumerate}
\end{theorem}
\begin{proof}
	By proposition \ref{prop2.1}, $rank(Z_g(j,\theta))$ is $\displaystyle{\frac{M}{q}}-$periodic with repect to $j$ and $1$-periodic with respect to $\theta$.  And by remark \ref{rem1}, $card(\mathcal{K}_j)$ is $\displaystyle{\frac{M}{q}}$-periodic. Hence $(2)\Longleftrightarrow (3)$.\\
	We will use the notations in lemma \ref{lem2.4}.  It is obvious that for $f\in \ell^2(\mathbb{Z})$, $F(j,\theta)=0$ for all $j\in \mathbb{N}_M$ and  a.e $\theta\in [0,1[$ if and only if $f=0$. Then by lemma \ref{lem2.4}, $(1)$ is equivalent to the fact that: for $f\in \ell^2(\mathbb{S})$,\\ ($\forall j\in \mathbb{N}_M$, a.e $\theta\in [0,1[$, $Z_g(j,\theta)F(j,\theta)=0)\Longrightarrow (\forall j\in \mathbb{N}_M$, a.e $\theta\in [0,1[,\; F(j,\theta)=0$).
	\begin{itemize}
		\item[] $(1)\Longrightarrow(3):$ Assume that $\mathcal{G}(g,L,M,N)$ is complete in $\ell^2(\mathbb{S})$ and suppose, by contradiction, that $(3)$ fails. Then by proposition \ref{prop2.1}, there exist $j_0\in \mathbb{N}_M$ and $E_0\subset [0,1[$ with positive measure such that for all $\theta\in E_0$:
		$$rank(Z_g(j_0,\theta))< card(\mathcal{K}_{j_0}).$$
		For a.e $\theta \in [0,1[$, denote $\mathbb{P}(j_0,\theta):\mathbb{C}^p\rightarrow \mathbb{C}^p$ the orthogonal projection onto the kernel of $Z_g(j_0,\theta)$. Let $\{e_k\}_{k\in \mathbb{N}_p}$ be the standard orthonormal basis of $\mathbb{C}^p$. 
		Suppose that $F=span\{e_k:\,k\in \mathcal{K}_j\}\subset N(\mathbb{P}(j_0,\theta))$ for some $\theta \in E_0$. Then $F\oplus N(Z_g(j_0,\theta))$ is an orthogonal sum. Thus:
		$$\begin{array}{rcl}
			p&\geqslant&dim(F\oplus N(Z_g(j_0,\theta))\,)\\
			&=&dim F+dim (N(Z_g(j_0,\theta))\,)\\
			&=& card(\mathcal{K}_j)+(p-rank(Z_g(j_0,\theta))\,).
		\end{array}$$
		Hence $rank(Z_g(j_0,\theta)\,)\geqslant card(\mathcal{K}_j).$ Contradiction. Then there exist $k_0\in \mathcal{K}_{j_0}$ and $E_0'\subset E_0$ with positive measure such that $e_{k_0}\notin N(\mathbb{P}(j_0,\theta)\,)$ for a.e $\theta\in E_0'$. i.e. $\mathbb{P}(j_0,\theta)e_{k_0}\neq 0$ for  $\theta \in E_0'$. Define for all $j\in \mathbb{N}_M$, a.e $\theta \in [0,1[$, $F(j,\theta)=\delta_{j,j_0}.\mathbb{P}(j_0,\theta)e_{k_0}.$ Observe that if $k\in \mathbb{N}_p-\mathcal{K}_{j_0}$, then $e_k\in N(Z_g(j_0,\theta)\,)$ for a.e $\theta \in [0,1[$. Then for all $k\in \mathbb{N}_p-\mathcal{K}_{j_0}$, a.e $\theta \in [0,1[$, $\mathbb{P}(j_0,\theta)e_k=e_k$. Thus for $k\in \mathbb{N}_p-\mathcal{K}_j$, $F(j_0,\theta)_k=\langle F(j_0,\theta),e_k\rangle=\langle \mathbb{P}(j_0,\theta)e_{k_0},e_k\rangle=\langle e_{k_0},\mathbb{P}(j_0,\theta) e_k\rangle=\langle e_{k_0},e_k\rangle=0.$\\
		Define $f\in \ell^2(\mathbb{Z})$ by $z_{pM}f(j+kM,\theta)=F(j,\theta)_k$ for all $j\in \mathbb{N}_M$, a.e $\theta\in [0,1[$. Then by lemma \ref{lem1.1}, $f\in \ell^2(\mathbb{S})$. Since $F(j_0,\theta)=\mathbb{P}(j_0,\theta)e_{k_0}\neq 0$ for all $\theta \in E_0'$ which is with positive measure, then $f\neq 0$.\\
		On the other hand, we have $Z_g(j,\theta)F(j,\theta)=0$ for all $j\in \mathbb{N}_M$ and a.e $\theta\in [0,1[$. In fact, if $j\neq j_0$, then, by definition of $F$, for a.e $\theta\in [0,1[$, $F(j,\theta)=0$, hence $Z_g(j,\theta)F(j,\theta)=0$. Otherwise, $F(j_0,\theta)=\mathbb{P}(j_0,\theta)e_{k_0}$, then $Z_g(j_0,\theta)F(j_0,\theta)=Z_g(j_0,\theta)\mathbb{P}(j_0,\theta)e_{k_0}=0$ since $\mathbb{P}(j_0,\theta)e_{k_0}\in N(Z_g(j_0,\theta)\,)$. Then, by lemma \ref{lem2.4}, $f$ is orthogonal to $\mathcal{G}(g,L,M,N)$ but $f\neq 0$. Contradiction with $(1)$.
		\item[]$(2)\Longrightarrow (1)$: Assume $(3)$ and let $f\in \ell^2(\mathbb{S})$ such that  for all $j\in \mathbb{N}_M$ and a.e $\theta\in [0,1[$, \begin{equation}Z_g(j,\theta)F(j,\theta)=0. 
		\end{equation}Let's prove that $F(j,\theta)=0$ for all $j\in \mathbb{N}_M$, a.e $\theta\in [0,1[$. Let $j\in \mathbb{N}_M$. If $\mathcal{K}_j=\emptyset$, then by the definition of $z_{pM}f$  and since $pM=qN$, $F(j,\theta)=0$ for a.e $\theta\in [0,1[$. Otherwise, i.e. $\mathcal{K}_j\neq \emptyset$. Let $k\in \mathbb{N}_p-\mathcal{K}_j$, then, by the definition of the Zak transform and since $pM=qN$, the $k$-th column of $Z_g(j,\theta)$ is identically zero and we also have $z_{pM}f(j+kM,\theta)=0$ for a.e $\theta \in [0,1[$. From $(3)$, the submatrix of $Z_g(j,\theta)$ of size $qL\times card(\mathcal{K}_j)$ obtained by removing all the columns with indices not in $\mathcal{K}_j$ has the same rank than $Z_g(j,\theta)$ which is  $card(\mathcal{K}_j)$. Then this submatrix is injective, thus, by equality (2), $z_{pM}f(j+kM,\theta)=0$ for a.e $\theta \in [0,1[$. Then for all $k\in \mathbb{N}_p$, $z_{pM}f(j+kM,\theta)=0$ for a.e $\theta \in [0,1[$. Hence $F(j,\theta)=0$ for a.e $\theta \in [0,1[$. Hence $\mathcal{G}(g,L,M,N)$ is complete in $\ell^2(\mathbb{S})$.
	\end{itemize}
\end{proof}
Now we characterize multi-window Gabor frames for $\ell^2(\mathbb{S})$ using the Zak transform.
\begin{theorem}\label{prop3.6}
	Given \(g:=\{g_l\}_{l\in \mathbb{N}_L} \subset \ell^2(\mathbb{S})\). Then the following statements are equivalent:
	\begin{enumerate}
		\item $\mathcal{G}(g,L, M,N)\) is a frame for \(\ell^2(S)$ with frame bounds $0 < A \leq B $.
		\item 
		\begin{equation}
			\frac{A}{M}.\mathcal{K}(j)
			\leq \sum_{l\in \mathbb{N}_L} Z_{g_l}^*(j, \theta)Z_{g_l}(j, \theta) \leq \frac{B}{M}.\mathcal{K}(j).
		\end{equation}
		for all $j \in \frac{M}{q}\) and a.e $\theta \in [0, 1[$.
		\item The inequality $(3)$ holds for all $j \in \mathbb{N}_M$ and a.e $\theta \in [0, 1[$.\\
	\end{enumerate}
\end{theorem}
For the proof, we will need the following lemma.
\begin{lemma}\label{lem3.7}
	Denote $L^{\infty}(\mathbb{S}_{pM}\times [0,1[)$ the set of functions $F$ on $\mathbb{S}_{pM}\times [0,1[$ such that for all $j\in \mathbb{S}_{pM}$, $F(j,.)\in L^{\infty}([0,1[)$, and  $\Delta := z_{pM \mid \ell^2(\mathbb{S})}^{-1}\left(L^{\infty}(\mathbb{S}_{pM} \times [0,1[)\right).$ Let $g:=\{g_l\}_{l\in \mathbb{N}_L}\subset \ell^2(\mathbb{S})$. Then the following statements are equivalent:
	\begin{enumerate}
		\item $\mathcal{G}(g,L,M,N)$ is a frame for $\ell^2(\mathbb{S})$ with frame bounds $A\leq B$.
		\item For all  $f \in \Delta$, we have:
		$$
		\begin{array}{rcl}
		\displaystyle{\frac{A}{M} \sum_{j=0}^{M-1} \int_{0}^{1} \|F(j, \theta)\|^2  d\theta} &\leq& \displaystyle{\sum_{l=0}^{l-1}\sum_{j=0}^{M-1} \int_{0}^{1} \| Z_{g_l}(j, \theta) F(j, \theta)\|^2  d\theta}\\
		& \leq& \displaystyle{\frac{B}{M} \sum_{j=0}^{M-1} \int_{0}^{1} \|F(j, \theta)\|^2 d\theta.}
		\end{array}$$
		Where $F(j,\theta)$, for all $j\in \mathbb{N}_M$ and a.e $\theta \in [0,1[$, is as defined in lemma \ref{lem2.4}.
	\end{enumerate}
\end{lemma}
\begin{proof}
	By density of $L^{\infty}(\mathbb{S}_{pM}\times [0,1[)$ in $L^2(\mathbb{S}_{pM}\times [0,1[)$ and by the unitarity of $z_{pM}$ from $\ell^2(\mathbb{S})$ onto $L^2(\mathbb{S}_{pM}\times [0,1[)$ (Since $pM=qN$, then $\mathbb{S}$ is $qN\mathbb{Z}$-periodic in $\mathbb{Z}$), $\Delta$ is dense in $\ell^2(\mathbb{S})$. Hence $\mathcal{G}(g,L,M,N)$ is a frame for $\ell^2(\mathbb{S})$ with frame bounds $A\leq B$ if and only if for all $f\in \Delta$, $A\|f\|^2\leq \displaystyle{\sum_{l=0}^{L-1}\sum_{n\in \mathbb{Z}}\sum_{m=0}^{M-1}\left\vert \langle f, e^{2\pi i \frac{m}{M}.}g_l(.-nN)\rangle\right\vert^2}\leq B\|f\|^2.$
	Let $f\in \Delta$, it is clear that  for all $r\in \mathbb{N}_q$, $Z_{g_l}(j,.)F(j,.)_r\in L^2([0,1[)$. 
	We have: $$\begin{array}{rcl}
		&&\displaystyle{\sum_{l=0}^{L-1}\sum_{n\in \mathbb{Z}}\sum_{m=0}^{M-1}\left\vert \langle f, e^{2\pi i \frac{m}{M}.}g_l(.-nN)\rangle\right\vert^2}\\
		&=&\displaystyle{\sum_{l=0}^{L-1}\sum_{r=0}^{q-1}\sum_{n\in \mathbb{Z}}\sum_{m=0}^{M-1}\left\vert \langle f, e^{2\pi i \frac{m}{M}.}g_l(.-(r+nq)N)\rangle\right\vert^2}\\
		&=&\displaystyle{\sum_{l=0}^{L-1}\sum_{r=0}^{q-1}\sum_{n\in \mathbb{Z}}\sum_{m=0}^{M-1}\left\vert \langle z_{pM}f, z_{pM}\left(e^{2\pi i \frac{m}{M}.}g_l(.-(r+nq)N)\right)\rangle\right\vert^2}\;\;\; \text{ by unitarity of }z_{pM}\\
		&=&\displaystyle{\sum_{l=0}^{L-1}\sum_{r=0}^{q-1}\sum_{n\in \mathbb{Z}}\sum_{m=0}^{M-1}\left\vert \langle z_{pM}f, z_{pM}\left(e^{2\pi i \frac{m}{M}.}g_l(.-rN-npM)\right)\rangle\right\vert^2} \;\;\; \text{ since }qN=pM.
	\end{array}$$
	By a simple calculation, we obtain that for all $j\in \mathbb{Z}$ and a.e $\theta\in \mathbb{R}$: 
	$$z_{pM}\left(e^{2\pi i \frac{m}{M}.}g_l(.-rN-npM)\right)(j,\theta)=z_{pM}g_l(j-rN,\theta)\, e^{2\pi i n\theta}\,e^{2\pi i \frac{m}{M}j}.$$
	Then: $$\begin{array}{rcl}
		&&\langle z_{pM}f, z_{pM}\left(e^{2\pi i \frac{m}{M}.}g_l(.-rN-npM)\right)\rangle\\
		&=&\displaystyle{\sum_{j=0}^{pM-1} \int_{0}^1 z_{pM}f(j,\theta)\, \overline{z_{pM}g_l(j-rN,\theta)}e^{-2\pi i n\theta}\, d\theta.\,e^{-2\pi i \frac{m}{M}j}}\\
		&=&\displaystyle{\sum_{j=0}^{M-1}\sum_{k=0}^{p-1} \int_{0}^1 z_{pM}f(j+kM,\theta)\, \overline{z_{pM}g_l(j+kM-rN,\theta)}e^{-2\pi i n\theta}\, d\theta.\,e^{-2\pi i \frac{m}{M}j}}\\
		&=&\displaystyle{\sum_{j=0}^{M-1}\int_{0}^1 \overline{Z_{g_l}(j,\theta)F(j,\theta)_{r}}\,e^{-2\pi i n\theta}d\theta.\, e^{-2\pi i\frac{m}{M}j}}\\
		&=& \displaystyle{\sum_{j=0}^{M-1}T(j).\, e^{-2\pi i\frac{m}{M}j}} \:\:\: \text{ where } T(j)=\displaystyle{\int_{0}^1 \overline{Z_{g_l}(j,\theta)F(j,\theta)_{r}}\,e^{-2\pi i n\theta}d\theta}\\
	\end{array}$$
	Observe that $T$ is $M$-periodic. Since $\{\displaystyle{\frac{1}{\sqrt{M}}e^{2\pi i \frac{m}{M}.}}\}_{m\in \mathbb{N}_M}$ is an orthonormal basis for $\ell^2(\mathbb{N}_M)$; the space of $M$-periodic sequences, then we have:$$\begin{array}{rcl}
		&&\displaystyle{\sum_{m=0}^{M-1}\left\vert \langle z_{pM}f, z_{pM}\left(e^{2\pi i \frac{m}{M}.}g_l(.-rN-npM)\right)\rangle \right\vert^2}\\
		&=& \displaystyle{\sum_{m=0}^{M-1}\left \vert \langle T, e^{2\pi i \frac{m}{M}.}\rangle \right\vert^2}\\
		&=& M\|T\|^2\\
		&=&\displaystyle{M\sum_{j=0}^{M-1}\left\vert\int_0^1 \overline{Z_{g_l}(j,\theta)F(j,\theta)_{r}}\,e^{-2\pi i n\theta}d\theta \right\vert^2}.\\
	\end{array}$$
	Since $\{e^{2\pi i n\theta}\}_{n\in \mathbb{Z}}$ is an orthonormal basis for $L^2([0,1[)$, then we have: 
	$$\begin{array}{rcl}
		&&\displaystyle{\sum_{n\in \mathbb{Z}}\sum_{m=0}^{M-1}\left\vert \langle z_{pM}f, z_{pM}\left(e^{2\pi i \frac{m}{M}.}g_l(.-rN-npM)\right)\rangle \right\vert^2}\\
		&=&\displaystyle{M\sum_{j=0}^{M-1}\sum_{n\in \mathbb{Z}}\left\vert\int_0^1 \overline{Z_{g_l}(j,\theta)F(j,\theta)_{r}}\,e^{-2\pi i n\theta}d\theta \right\vert^2}.\\
		&=&\displaystyle{M\sum_{j=0}^{M-1}\sum_{n\in \mathbb{Z}}\left\vert \langle \overline{Z_{g_l}(j,.)F(j,.)_{r}},e^{2\pi i n.}\rangle \right\vert^2}.\\
		&=&\displaystyle{M\sum_{j=0}^{M-1}\left\| \overline{Z_{g_l}(j,.)F(j,.)_r}\right\|^2}\\
		&=&\displaystyle{M\sum_{j=0}^{M-1}\int_0^1 \left\vert Z_{g_l}(j,\theta)F(j,\theta)_r\right\vert^2\,d\theta}\\
	\end{array}$$
	Hence: $$\begin{array}{rcl}
		&&\displaystyle{\sum_{r=0}^{q-1}\sum_{n\in \mathbb{Z}}\sum_{m=0}^{M-1}\left\vert \langle z_{pM}f, z_{pM}\left(e^{2\pi i \frac{m}{M}.}g_l(.-rN-npM)\right)\rangle \right\vert^2}\\
		&=&\displaystyle{M\sum_{j=0}^{M-1}\int_0^1 \sum_{r=0}^{q-1}\left\vert Z_{g_l}(j,\theta)F(j,\theta)_r\right\vert^2\,d\theta}\\
		&=&\displaystyle{M\sum_{j=0}^{M-1}\int_0^1 \left\| Z_{g_l}(j,\theta)F(j,\theta)\right\|^2\,d\theta}.\\
	\end{array}$$
	The norm in the last line is the 2-norm in $\mathbb{C}^q$. Thus:
	 $$\begin{array}{rcl}
	 &&\displaystyle{\sum_{r=0}^{q-1}\sum_{n\in \mathbb{Z}}\sum_{m=0}^{M-1}\left\vert \langle z_{pM}f, z_{pM}\left(e^{2\pi i \frac{m}{M}.}g_l(.-rN-npM)\right)\rangle \right\vert^2}\\ &=&\displaystyle{M\sum_{l=0}^{L-1}\sum_{j=0}^{M-1}\int_0^1 \left\| Z_{g_l}(j,\theta)F(j,\theta)\right\|^2\,d\theta}.
	 \end{array}$$
	Hence: \begin{equation}
		\displaystyle{\sum_{l=0}^{L-1}\sum_{n\in \mathbb{Z}}\sum_{m=0}^{M-1}\left\vert \langle f, e^{2\pi i \frac{m}{M}.}g_l(.-nN)\rangle\right\vert^2}=\displaystyle{M\sum_{l=0}^{L-1}\sum_{j=0}^{M-1}\int_0^1 \left\| Z_{g_l}(j,\theta)F(j,\theta)\right\|^2\,d\theta}.
	\end{equation}
	On the other hand, we have by unitarity of $z_{pM}$:$$\begin{array}{rcl}
		\|f\|^2&=&\|z_{pM}f\|^2\\
		&=&\displaystyle{\sum_{j=0}^{pM-1}\int_0^1\vert z_{pM}f(j,\theta)\vert^2\, d\theta}\\
		&=&\displaystyle{\sum_{j=0}^{M-1}\sum_{k=0}^{p-1}\int_0^1 \vert z_{pM}(j+kM,\theta)\vert^2\, d\theta}\\
		&=&\displaystyle{\sum_{j=0}^{M-1}\int_0^1 \sum_{k=0}^{p-1} \vert F(j,\theta)_k\vert^2\, d\theta}\\
		&=&\displaystyle{\sum_{j=0}^{M-1}\int_0^1 \| F(j,\theta)\|^2\, d\theta}\\
	\end{array}$$
	The norm in the last line is the 2-norm in $\mathbb{C}^p$. Thus: \begin{equation}
		\|f\|^2=\displaystyle{\sum_{j=0}^{M-1}\int_0^1 \| F(j,\theta)\|^2\, d\theta}.
	\end{equation}
	Then, combining $(5)$ and $(6)$, the proof is completed.
\end{proof}

\begin{proof}[Proof of Theorem \ref{prop3.6}]\hspace{1cm}
	\begin{enumerate}
		\item[]$(1)\Longrightarrow (3)$: Assume that $\mathcal{G}(g,L,M,N)$ is a frame for $\ell^2(\mathbb{S})$. Then for all $f\in \Delta$, \begin{equation*}
\begin{array}{rcl}
\displaystyle{\frac{A}{M} \sum_{j=0}^{M-1} \int_{0}^{1} \|F(j, \theta)\|^2 \, d\theta} &\leq& \displaystyle{\sum_{l=0}^{l-1}\sum_{j=0}^{M-1} \int_{0}^{1} \| Z_{g_l}(j, \theta) F(j, \theta)\|^2 \, d\theta } \\
&\leq& \displaystyle{\frac{B}{M} \sum_{j=0}^{M-1} \int_{0}^{1} \|F(j, \theta)\|^2 \, d\theta.}
		\end{array}	
		\end{equation*}
		Fix $x:=\{x_k\}_{k\in \mathbb{N}_p}\in  \mathbb{C}^p$, $j_0\in \mathbb{N}_M$ and $h\in L^{\infty}([0,1[)$, and define for all $j\in \mathbb{N}_M$ and a.e $\theta\in [0,1[$, $F(j,\theta):=\{\delta_{j,j_0}\,\chi_{\mathcal{K}_j}(k)\,x_k\, h(\theta) \}_{k\in \mathbb{N}_p}$. \\Then: $\displaystyle{ \sum_{j=0}^{M-1} \int_{0}^{1} \|F(j, \theta)\|^2 \, d\theta=\|\mathcal{K}(j_0)x\|^2\int_0^1 \vert h(\theta)\vert^2 \,d\theta=}$\\ $\displaystyle{\langle \mathcal{K}(j_0)x,x\rangle \int_0^1 \vert h(\theta)\vert^2\,d\theta}$ (lemma \ref{lem2.9}). On the other hand, we have: $$
		\begin{array}{rcl}
			\displaystyle{\sum_{l=0}^{l-1}\sum_{j=0}^{M-1} \int_{0}^{1} \| Z_{g_l}(j, \theta) F(j, \theta)\|^2 \, d\theta }&=&\displaystyle{\sum_{l=0}^{l-1} \int_{0}^{1} \| Z_{g_l}(j_0, \theta) F(j_0, \theta)\|^2 \, d\theta} \\
			&=&\displaystyle{\sum_{l=0}^{l-1} \int_{0}^{1} \sum_{r=0}^{q-1} \left\vert \left(Z_{g_l}(j_0, \theta) F(j_0, \theta)\right)_r \,\right\vert^2 \, d\theta} \\
			&=&\displaystyle{\sum_{l=0}^{l-1} \int_{0}^{1} \sum_{r=0}^{q-1} \left\vert \sum_{k=0}^{p-1} Z_{g_l}(j_0, \theta)_{r,k} F(j_0, \theta)_k \,\right\vert^2 \, d\theta} \\
			&=&\displaystyle{\sum_{l=0}^{l-1} \int_{0}^{1} \sum_{r=0}^{q-1} \left\vert \sum_{k=0}^{p-1} Z_{g_l}(j_0, \theta)_{r,k} \,\chi_{\mathcal{K}_{j_0}}(k)\,x_k\, h(\theta)\,\right\vert^2 d\theta }\\
			&=&\displaystyle{\sum_{l=0}^{l-1} \int_{0}^{1} \sum_{r=0}^{q-1}\left \vert \sum_{k=0}^{p-1} Z_{g_l}(j_0, \theta)_{r,k} (\mathcal{K}(j_0)x)_k \right\vert^2 \vert h(\theta)\vert^2 d\theta} \\
			&=&\displaystyle{\sum_{l=0}^{l-1} \int_{0}^{1} \sum_{r=0}^{q-1} \left\vert \left(Z_{g_l}(j_0, \theta)\mathcal{K}(j_0)x\right)_r \right\vert^2 \vert h(\theta)\vert^2 d\theta}\\
			&=&\displaystyle{\sum_{l=0}^{l-1} \int_{0}^{1} \left\|Z_{g_l}(j_0,\theta)\mathcal{K}(j_0)x \right\|^2 \vert h(\theta)\vert^2 d\theta} \\
			&=&\displaystyle{\sum_{l=0}^{l-1} \int_{0}^{1} \left\|Z_{g_l}(j_0,\theta)x \,\right\|^2 \vert h(\theta)\vert^2  d\theta}\;\;\; \text{ by lemma }\ref{lem2.8} \\
			&=&\displaystyle{\int_{0}^{1} \left\langle \sum_{l=0}^{l-1} Z_{g_l}(j_0,\theta)^*Z_{g_l}(j_0,\theta)x,x\right\rangle \vert h(\theta)\vert^2 d\theta}. \\
		\end{array}$$
		Then for all $j\in \mathbb{N}_M$,  $x\in \mathbb{C}^p$ and $h\in L^{\infty}([0,1[)$, we have: 
		\begin{equation}
		\begin{array}{rcl}
\displaystyle{\frac{A}{M}.\langle \mathcal{K}(j)x,x\rangle \,\int_0^1\vert h(\theta)\vert^2 \,d\theta} &\leq& \displaystyle{\int_{0}^{1} \left\langle \sum_{l=0}^{l-1} Z_{g_l}(j,\theta)^*Z_{g_l}(j,\theta)x,x\right\rangle \,\vert h(\theta)\vert^2 \, d\theta} \\
&\leq& \displaystyle{\frac{B}{M}.\langle \mathcal{K}(j)x,x\rangle \int_0^1\vert h(\theta)\vert^2\, d\theta.}
	\end{array}
		\end{equation}
		For $j\in \mathbb{N}_M$ and $x\in \mathbb{C}^p$ fixed, denote $C=\displaystyle{\frac{A}{M}.\langle \mathcal{K}(j)x,x\rangle}$ and $D=\displaystyle{\frac{B}{M}.\langle \mathcal{K}(j)x,x\rangle}$. Assume, by contradiction, that \begin{equation} C>\displaystyle{\left\langle \sum_{l=0}^{L-1}Z_{g_l}(j,\theta)^*Z_{g_l}(j,\theta)x,x\right\rangle} ,
		\end{equation} on a subset of $[0,1[$ with a positive measure. Denote 
		$D=\{\theta \in [0,1[:\; (8) \text{ holds}\}.$\\
		For all $k\in \mathbb{N}$, denote: $$D_k:=\left\{\theta\in [0,1[:\; C-\displaystyle{\frac{C}{k+1}}\leq \displaystyle{\langle \sum_{l=0}^{L-1}Z_{g_l}(j,\theta)^*Z_{g_l}(j,\theta)x,x\rangle} \leq C-\displaystyle{\frac{C}{k}}\right\}.$$
		It is clear that $\{D_k\}_{k\in \mathbb{N}}$ forms a partition for $D$. Since $mes(D)>0$, then there exists $k\in \mathbb{N}$ such that $mes(D_k)>0$. Let $h:=\chi_{D_k}$, we have: $$\begin{array}{rcl}
			\displaystyle{\int_{0}^{1} \left\langle \sum_{l=0}^{l-1} Z_{g_l}(j,\theta)^*Z_{g_l}(j,\theta)x,x\right\rangle \,\vert h(\theta)\vert^2 \, d\theta}&=&\displaystyle{\int_{D_k} \left\langle \sum_{l=0}^{l-1} Z_{g_l}(j,\theta)^*Z_{g_l}(j,\theta)x,x\right\rangle d\theta}\\
			&\leq&(C-\displaystyle{\frac{C}{k}}). mes(D_k)\\
			&<&C.mes(D_k)\\
			&=&C\displaystyle{\int_0^1 \vert h(\theta)\vert^2\,d\theta}.
		\end{array}$$
	 Contradiction with $(7)$. 
		Suppose, again by contradiction, that \begin{equation} D<\displaystyle{\left\langle \sum_{l=0}^{L-1}Z_{g_l}(j,\theta)^*Z_{g_l}(j,\theta)x,x\right\rangle}, 
		\end{equation} on a subset of $[0,1[$ with a positive measure. Denote 
		$D'=\{\theta \in [0,1[:\; (9) \text{ holds}\}.$\\
		For all $k\in \mathbb{N}$, $m\in \mathbb{N}$, denote $$D_{k,m}':=\left\{\theta\in [0,1[:\; D(k+\displaystyle{\frac{1}{m+1}})\leq \displaystyle{\langle \sum_{l=0}^{L-1}Z_{g_l}(j,\theta)^*Z_{g_l}(j,\theta)x,x\rangle} \leq D(k+\displaystyle{\frac{1}{m}})\right\}.$$
		It is clear that $\{D_{k,m}'\}_{k\in \mathbb{N}}$ forms a partition for $D'$. Since $mes(D')>0$, then there exist $k\in \mathbb{N}$ and $m\in \mathbb{N}$ such that $mes(D_{k,m}')>0$. Let $h:=\chi_{D_{k,m}'}$, we have: $$\begin{array}{rcl}
			\displaystyle{\int_{0}^{1} \langle \sum_{l=0}^{l-1} Z_{g_l}(j,\theta)^*Z_{g_l}(j,\theta)x,x\rangle \,\vert h(\theta)\vert^2 \, d\theta}&=&\displaystyle{\int_{D_{k,m}'} \langle \sum_{l=0}^{l-1} Z_{g_l}(j,\theta)^*Z_{g_l}(j,\theta)x,x\rangle \, d\theta}\\
			&\geqslant&D(k+\displaystyle{\frac{1}{m+1}}). mes(D_{k,m}')\\
			&>&D.mes(D_{k,m}')\\
			&=&D\displaystyle{\int_0^1 \vert h(\theta)\vert^2\,d\theta}.
		\end{array}$$
		Contradiction with $(7)$. Hence for all $j\in \mathbb{N}_M$, and a.e $\theta \in [0,1[$, we have: 
		\begin{equation*}
			\frac{A}{M}.\mathcal{K}(j)
			\leq \sum_{l\in \mathbb{N}_L} Z_{g_l}^*(j, \theta)Z_{g_l}(j, \theta) \leq \frac{B}{M}.\mathcal{K}(j).
		\end{equation*}
		\item[]$(3)\Longrightarrow(1)$: Assume $(3)$. For $f\in \Delta$, $j\in \mathbb{N}_M$ and a.e $\theta\in [0,1[$,  the $k$-th component of $F(j,\theta)$ is zero if $k\notin \mathcal{K}_j$. Then $\mathcal{K}(j)F(j,\theta)=F(j,\theta)$. Then: \begin{equation*}
			\frac{A}{M}\|F(j,\theta)\|^2\leq \left \langle \sum_{l=0}^{L-1} Z_{g_l}(j,\theta)^*Z_{g_l}(j,\theta)F(j,\theta),F(j,\theta)\right \rangle\leq \frac{B}{M}\|F(j,\theta)\|^2.
		\end{equation*}
		Hence: \begin{equation*}
			\frac{A}{M}\sum_{j=0}^{M-1}\int_0^1\|F(j,\theta)\|^2\leq \sum_{l=0}^{L-1}\sum_{j=0}^{M-1}\int_0^1 \| Z_{g_l}(j,\theta)F(j,\theta)\|^2\leq \frac{B}{M}\sum_{j=0}^{M-1}\int_0^1\|F(j,\theta)\|^2.
		\end{equation*}
		Then lemma \ref{lem3.7} implies $(1)$.
		\item[]$(3)\Longrightarrow (2)$: Since $\mathbb{N}_{\frac{M}{q}}\subset \mathbb{N}_M$.
		\item[]$(2)\Longrightarrow (3)$: Assume $(2)$. Then the inequality $(3)$ holds for all $ j \in \mathbb{N}_{\frac{M}{q}}$ and a.e. $ \theta \in [0, 1[$. Let's prove that it  holds for all $ j \in \mathbb{N}_M$ and a.e. $\theta \in [0, 1[$. Let $j \in \mathbb{N}_M$. By Lemma \ref{lem2.5}, there exists a unique $(j', r', k', \ell') \in \mathbb{N}_{\frac{M}{q}} \times \mathbb{N}_q \times \mathbb{N}_p \times \mathbb{Z}$ such that $j = j' + k'M - r'N + \ell'qN$. Then, by the
		quasiperiodicity of the discrete Zak transform, we have, for all $l\in \mathbb{N}_L$, after a simple calculation:
		
		$(Z_{g_l}(j, \theta)^*Z_{g_l}(j, \theta))_{k_1, k_2} =$\\
		$\displaystyle{\sum_{r=0}^{q-1}
		(\overline{Z_{qN} g_l) (j' + (k_1 + k')M - (r + r')N, \theta)} (Z_{qN}g) (j' + (k_2 + k')M - (r + r')N, \theta)}$
		\begin{equation*}
			=
			\begin{cases} 
				(Z_{g_l}(j', \theta)^*Z_{g_l}(j', \theta))_{k_1+k', k_2+k'}, & \text{if } k_1 + k' < p, \; k_2 + k' < p \\
				e^{-2\pi i\theta} (Z_{g_l}(j', \theta)^*Z_{g_l}(j', \theta))_{k_1+k', k_2+k'-p}, & \text{if } k_1 + k' < p, \; k_2 + k' \geq p \\
				e^{2\pi i\theta} (Z_{g_l}(j', \theta)^*Z_{g_l}(j', \theta))_{k_1+k'-p, k_2+k'}, & \text{if } k_1 + k' \geq p, \; k_2 + k' < p \\
				(Z_{g_l}(j', \theta)^*Z_{g_l}(j', \theta))_{k_1+k'-p, k_2+k'-p}, & \text{if } k_1 + k' \geq p, \; k_2 + k' \geq p.
			\end{cases}
		\end{equation*}
		for $k_1, k_2 \in \mathbb{N}_p$ and a.e. $\theta \in [0, 1[$.\\
		Define \( V: \mathbb{C}^p \to \mathbb{C}^p \) by $Vx = y = (y_0, y_1, \ldots, y_{p-1})^t$, where:
		
		$$
		y_k =
		\begin{cases} 
			e^{-2\pi i \theta}x_{k-k'+p}, & \text{if } 0 \leq k < k' \\
			x_{k-k'}, & \text{if } k' \leq k < p,
		\end{cases}
		$$
		for $x \in \mathbb{C}^p$. Then $V$ is a unitary operator, and
		
		$$
		\langle Z_{g_l}(j, \theta)^*Z_{g_l}(j, \theta)x, x \rangle = \langle Z_{g_l}(j', \theta)^*Z_{g_l}(j', \theta)Vx, Vx \rangle,
		$$
		for a.e. $\theta \in [0, 1[$ and all $x \in \mathbb{C}^p$. Then, by $(2)$, we have:
		
		\begin{equation}
			\frac{A}{M} \langle V^*\mathcal{K}(j') Vx, x \rangle
			\leq \left\langle \sum_{l=0}^{L-1}Z_{g_l}(j, \theta)^*Z_{g_l}(j, \theta)x, x \right\rangle
			\leq \frac{B}{M}\langle V^*\mathcal{K}(j') Vx, x \rangle,
		\end{equation}
		for a.e. \( \theta \in [0, 1) \) and each $x \in \mathbb{C}^p$. When $k+k' < p$, $k+k' \in \mathcal{K}_{j'}$ if and only if $j'+(k+k')M \in\mathbb{S}$,  equivalently, $j' + k'M - r'N + \ell'qN + kM \in\mathbb{S}$, i.e. $j + kM \in \mathbb{S}$. Therefore, $k + k' \in \mathcal{K}_j$ if and only if $k \in \mathcal{K}_j$ when $k + k' < p$. Similarly, $k + k' - p \in \mathcal{K}_{j'}$ if and only if $k \in \mathcal{K}_j$ when
		$k + k' \geq p$. It follows that:
		$$
		V^*\mathcal{K}(j')V = \mathcal{K}(j).
		$$
		Which together with $(10)$ gives $(3)$. The proof is completed.
	\end{enumerate}
\end{proof}

\begin{remark}
	In the case of $\mathbb{S}=\mathbb{Z}$. For all $j\in \mathbb{N}_{\frac{M}{q}}$, $\mathcal{K}_j=\mathbb{N}_p$. Then the condition $(2)$ in the Theorem \ref{prop3.5} is equivalent to $rank(Z_g(j,\theta)\,)=p$ for all $j\in \mathbb{N}_{\frac{M}{q}}$ and a.e $\theta \in [0,1[$. And the condition $(2)$ in the Theorem $\ref{prop3.6}$ is equivalent to: For all $j\in \mathbb{N}_{\frac{M}{q}}$ and a.e $\theta \in [0,1[$, 
	\begin{equation*}
		\frac{A}{M}.I_{p,p}
		\leq \sum_{l\in \mathbb{N}_L} Z_{g_l}^*(j, \theta)Z_{g_l}(j, \theta) \leq \frac{B}{M}.I_{p,p}.
	\end{equation*}
	
	Where $I_{p,p}$ is the identity matrix in $\mathcal{M}_{p,p}$.
		
\end{remark}
\section{Admissibility conditions for a complete multiwindow Gabor system and a multiwindow Gabor frame}

Note that, in what follows, we will use all the notations already introduced Without introducing them again. In this section, we study conditions for a periodic set $\mathbb{S}$ to admit a complete multi-window Gabor system, and a multi-window Gabor frame. Let $L,M,N\in \mathbb{N}$ and $p,q\in \mathbb{N}$ such that  $\displaystyle{\frac{N}{M} = \frac{p}{q}}$  and $p\wedge q=1$.

In what follows, we give some useful lemmas for the rest.
\begin{lemma}\label{lem4.1}\hspace{1cm}
	\begin{enumerate}
		\item $card(\mathcal{K}_j)\leq \,qL$ for all $j\in \mathbb{N}_{\frac{M}{q}}$ $\Longrightarrow$ $card(\mathbb{S}_N)\leq\, LM$.
		\item Assume that $card(\mathcal{K}_j)\leq \,qL$ for all $j\in \mathbb{N}_{\frac{M}{q}}$. Then:
		$$card(\mathcal{K}_j)=qL \text{ for all }j\in \mathbb{N}_{\frac{M}{q}}\;\Longleftrightarrow \;card(\mathbb{S}_N)=\, LM.$$
	\end{enumerate}
\end{lemma}
\begin{proof}\hspace{1cm}
	\begin{enumerate}
		\item Assume  that $card(\mathcal{K}_j)\leq \,qL$ for all $j\in \mathbb{N}_{\frac{M}{q}}$.  We have: 
		$$\begin{array}{rcl}
			card(\mathbb{S}_N)&=&\displaystyle{\sum_{j\in \mathbb{Z}}\chi_{\mathbb{S}_N}(j)}\\
			&=&\displaystyle{\sum_{j\in \mathbb{N}_{\frac{M}{q}}}\sum_{n\in \mathbb{Z}}\chi_{\mathbb{S}_N}(j+\displaystyle{\frac{M}{q}}n)}\\
			&=&\displaystyle{\sum_{j\in \mathbb{N}_{\frac{M}{q}}}card(\mathcal{K}_j)}\;\;\; \text{ by remark }\ref{rem1}\\
			&\leq& \displaystyle{\frac{M}{q}.qL=LM}.
		\end{array}$$
		\item Assume that $card(\mathcal{K}_j)\leq qL$ for all $j\in \mathbb{N}_{\frac{M}{q}}$.\\
		Assume that $card(\mathcal{K}_j)=qL$ for all $j\in \mathbb{N}_{\frac{M}{q}}$. Then by the proof of (1), we have: $$\begin{array}{rcl}
			card(\mathbb{S}_N)&=&\displaystyle{\sum_{j\in \mathbb{N}_{\frac{M}{q}}}card(\mathcal{K}_j)}\\
			&=&\displaystyle{\frac{M}{q}.qL=LM}.
		\end{array}$$
		Conversely, assume that $card(\mathbb{S}_N)=\, LM$. Again by the proof of (1), we have $\displaystyle{\sum_{j\in \mathbb{N}_{\frac{M}{q}}}card(\mathcal{K}_j)}=card(\mathbb{S}_N)=LM=\displaystyle{\frac{M}{q}.qL}$. Since 
		$card(\mathcal{K}_j)\leq \,qL$ for all $j\in \mathbb{N}_{\frac{M}{q}}$, then $card(\mathcal{K}_j)=\,qL$ for all $j\in \mathbb{N}_{\frac{M}{q}}$.
	\end{enumerate}
\end{proof}

\begin{lemma}\label{lem4.2}
	Let $L,M\in \mathbb{N}$ and $E_0,\,E_1,\, \ldots,\,E_{L-1}\subset \mathbb{Z}$ be mutually disjoint. Denote $E=\displaystyle{\bigcup_{l\in \mathbb{N}_L}E_l}$. Then the following statements are equivalent:
	\begin{enumerate}
		\item $\{e^{2\pi i \frac{m}{M}.}\chi_{E_l}\}_{m\in \mathbb{N}_M, \, l\in \mathbb{N}_L}$ is complete in  $\ell^2(E)$.
		\item For all $l\in \mathbb{N}_L$, $\{e^{2\pi i \frac{m}{M}.}\chi_{E_l}\}_{m\in \mathbb{N}_M}$ is complete in $\ell^2(E_l)$.
	\end{enumerate}
\end{lemma}
\begin{proof}\hspace{1cm}
	\begin{enumerate}
		\item[]$(1)\Longrightarrow (2)$: Assume $(1)$. Fix $l_0\in \mathbb{N}_L$ and let $f\in \ell^2(E_{l_0})$ be orthogonal to  $\{e^{2\pi i \frac{m}{M}.}\chi_{E_{l_0}}\}_{m\in \mathbb{N}_M}$. 
		Define $\overline{f}\in \ell^2(E)$ by $\overline{f}(j)=f(j)$ if $j\in E_{l_0}$ and $0$ otherwise. It is clear that if $l\neq l_0$, $\overline{f}$ is orthogonal to  $\{e^{2\pi i \frac{m}{M}.}\chi_{E_l}\}_{m\in \mathbb{N}_M}$. And we have: 
		$$\begin{array}{rcl}
			\langle \overline{f},e^{2\pi i \frac{m}{M}.}\chi_{E_{l_0}}\rangle&=&\displaystyle{\sum_{j\in E}\overline{f}(j)e^{-2\pi i \frac{m}{M}j}\chi_{E_{l_0}}(j)}\\
			&=& \displaystyle{\sum_{j\in E_{l_0}}f(j)e^{-2\pi i \frac{m}{M}j}\chi_{E_{l_0}}(j)}\\
			&=&0\;\;\; \text{ since } f\text{ is orthogonal to } \{e^{2\pi i \frac{m}{M}.}\chi_{E_{l_0}}\}_{m\in \mathbb{N}_M}.
		\end{array}$$
		Then $\overline{f}$ is orthogonal to $\{e^{2\pi i \frac{m}{M}.}\chi_{E_l}\}_{m\in \mathbb{N}_M, l\in \mathbb{N}_L}$ which is complete in $\ell^2(E)$, thus $\overline{f}=0$ on $E$, and then $f=0$ on $E_l$.
		
		\item[]$(2)\Longrightarrow(1)$: Assume $(3)$ and let $h\in \ell^2(E)$ be orthogonal to $\{e^{2\pi i \frac{m}{M}.}\chi_{E_l}\}_{m\in \mathbb{N}_M, l\in \mathbb{N}_L}$. For all $l\in \mathbb{N}_L$, define $h_l\in \ell^2(E_l)$ as the restriction of $h$ on $E_l$, i.e.  $h_l:=h|_{E_l}$. Let $l\in \mathbb{N}_L$. Fix $l\in \mathbb{N}_L$. Since $h$ is orthogonal to $\{e^{2\pi i \frac{m}{M}.}\chi_{E_l}\}_{m\in \mathbb{N}_M}$, then $\displaystyle{\sum_{j\in E}h(j) e^{-2\pi i \frac{m}{M}.}\chi_{E_l}(j)=0}$, then $\displaystyle{\sum_{j\in E_l}h(j) e^{-2\pi i \frac{m}{M}.}\chi_{E_l}(j)=0}$, thus $\displaystyle{\sum_{j\in E_l}h_l(j) e^{-2\pi i \frac{m}{M}.}\chi_{E_l}(j)=\langle h_l,e^{2\pi i \frac{m}{M}.}\chi_{E_l}\rangle =0}$. Hence $h_l$ is orthogonal to  $\{e^{2\pi i \frac{m}{M}.}\chi_{E_l}\}_{m\in \mathbb{N}_M}$ which is complete in $\ell^2(E_l)$. Hence $h_l=0$ on $E_l$. This for all $l\in \mathbb{N}_L$, therfore $h=0$ on $E$. Hence $\{e^{2\pi i \frac{m}{M}.}\chi_{E_l}\}_{m\in \mathbb{N}_M,l\in \mathbb{N}_L}$ is complete in $\ell^2(E)$.
		
	\end{enumerate}
	
\end{proof}

\begin{lemma}\label{lem4.3}
	Let $L,M\in \mathbb{N}$ and $E_0,\,E_1,\, \ldots,\,E_{L-1}\subset \mathbb{Z}$  be mutually disjoint. Denote $E=\displaystyle{\bigcup_{l\in \mathbb{N}_L}E_l}$. Then the following statements are equivalent:
	\begin{enumerate}
		\item $\{e^{2\pi i \frac{m}{M}.}\chi_{E_l}\}_{m\in \mathbb{N}_M, \, l\in \mathbb{N}_L}$ is a tight frame for $\ell^2(E)$ with frame bound $M$.
		\item For all $l\in \mathbb{N}_L$, $\{e^{2\pi i \frac{m}{M}.}\chi_{E_l}\}_{m\in \mathbb{N}_M}$ is a tight frame for $\ell^2(E_l)$ with frame bound $M$.
	\end{enumerate}
\end{lemma}
\begin{proof}\hspace{1cm}
	\begin{enumerate}
		\item[]$(1)\Longrightarrow (2)$: Assume $(1)$. Fix $l_0\in \mathbb{N}_L$ and let $f\in \ell^2(E_{l_0})$. Define $\overline{f}\in \ell^2(E)$ by $\overline{f}(j)=f(j)$ if $j\in E_{l_0}$ and $0$ otherwise. It is clear that $\langle \overline{f},e^{2\pi i \frac{m}{M}.}\chi_{E_l}\rangle=0$ if $l\neq l_0$ and that $\|\overline{f}\|=\|f\|$. Together with the fact that $\{e^{2\pi i \frac{m}{M}.}\chi_{E_l}\}_{m\in \mathbb{N}_M, \, l\in \mathbb{N}_L}$ is a tight frame for $\ell^2(E)$ with frame bound $M$, we have: $$
		\begin{array}{rcl}
			M\|f\|^2&=&\displaystyle{\sum_{m=0}^{M-1}\left\vert \sum_{j\in E}\overline{f}(j)e^{-2\pi i \frac{m}{M}j}\chi_{E_{l_0}}(j)\right\vert^2}\\
			&=&\displaystyle{\sum_{m=0}^{M-1}\left\vert\sum_{j\in E_{l_0}}f(j)e^{-2\pi i \frac{m}{M}j}\chi_{E_{l_0}}(j)\right\vert^2}\\
			&=&\displaystyle{\sum_{m=0}^{M-1}\left\vert\langle f,e^{2\pi i \frac{m}{M}.}\chi_{E_{l_0}}\rangle\right\vert^2}.
		\end{array}$$
		Hence $\{e^{2\pi i \frac{m}{M}.}\chi_{E_{l_0}}\}_{m\in \mathbb{N}_M}$ is a tight frame for $\ell^2(E_{l_0})$ with frame bound $M$. And this for all $l_0\in \mathbb{N}_L$.
		\item[]$(2)\Longrightarrow (1)$: Assume $(2)$. Let $h\in \ell^2(E)$. For all $l\in \mathbb{N}_L$, define $h_l\in \ell^2(E_l)$ as the restriction of $h$ en $E_l$, i.e $h_l=h|_{E_l}$. We have $\displaystyle{\sum_{j\in E}\vert h(j)\vert^2=\sum_{l\in \mathbb{N}_L}\sum_{j\in E_l} \vert h_l(j)\vert^2}$, then $\|h\|^2=\displaystyle{\sum_{l\in \mathbb{N}_L}\|h_l\|^2}.$ It is also clear that  $\langle h,e^{2\pi i \frac{m}{M}.}\chi_{E_l}\rangle=\langle h_l,e^{2\pi i \frac{m}{M}.}\chi_{E_l}\rangle$. Since For all $l\in \mathbb{N}_L$, $\{e^{2\pi i \frac{m}{M}.}\chi_{E_l}\}_{m\in \mathbb{N}_M}$ is a tight frame for $\ell^2(E_l)$ with frame bound $M$, then for all $l\in \mathbb{N}_l$, we have $M\|h_l\|^2=\displaystyle{\sum_{m\in \mathbb{N}_M}\vert \langle h_l,e^{2\pi i \frac{m}{M}.}\chi_{E_l}\rangle \vert^2}.$ Hence: 
		$$M\|h\|^2=\displaystyle{\sum_{l\in \mathbb{N}_L}\sum_{m\in \mathbb{N}_M}\vert \langle h,e^{2\pi i \frac{m}{M}.}\chi_{E_l}\rangle \vert^2}.$$
		This for all $h\in \ell^2(E_l)$. Hence $\{e^{2\pi i \frac{m}{M}.}\chi_{E_l}\}_{m\in \mathbb{N}_M, \, l\in \mathbb{N}_L}$ is a tight frame for $\ell^2(E)$ with frame bound $M$.
	\end{enumerate}
\end{proof}
\begin{lemma}\label{lem4.4}
	Let $L,M\in \mathbb{N}$ and $E_0,\,E_1,\, \ldots,\,E_{L-1}\subset \mathbb{Z}$ be mutually disjoint. Denote $E=\displaystyle{\bigcup_{l\in \mathbb{N}_L}E_l}$. Then the following statements are equivalent:
	\begin{enumerate}
		\item $\{e^{2\pi i \frac{m}{M}.}\chi_{E_l}\}_{m\in \mathbb{N}_M, \, l\in \mathbb{N}_L}$ is a tight frame for $\ell^2(E)$ with frame bound $M$.
		\item $\{e^{2\pi i \frac{m}{M}.}\chi_{E_l}\}_{m\in \mathbb{N}_M, \, l\in \mathbb{N}_L}$ is complete in  $\ell^2(E)$.
		\item For all $l\in \mathbb{N}_L$, $E_l$ is $M\mathbb{Z}$-congruent to a subset of $\mathbb{N}_M$.
		\item For all $l\in \mathbb{N}_L$, $\displaystyle{\sum_{k\in \mathbb{Z}}\chi_{E_l}(.+kM)\leq 1}$ on $\mathbb{Z}$.
	\end{enumerate}
\end{lemma}
\begin{proof}
	It is a direct result of lemma \ref{lem2.6}, lemma \ref{lem4.2} and lemma \ref{lem4.3} together.
\end{proof}

The following proposition presents a characterization for the admissibility of $\mathbb{S}$ to admit a complete multi-window Gabor system $\mathcal{G}(g, L, M, N)$.
\begin{theorem}\label{prop4.5}
	Then the following statements are equivalent:
	\begin{enumerate}
		\item There exist $g:=\{g_l\}_{l\in \mathbb{N}_L}\subset \ell^2(\mathbb{S})$ such that $\mathcal{G}(g,L,M,N)$ is complete in $\ell^2(\mathbb{S})$.
		\item For all $j\in \mathbb{N}_{\frac{M}{q}}$, we have: 
		\begin{equation}
			card(\mathcal{K}_j)\leq qL.
		\end{equation} 
		\item $(10)$ holds for all $j\in \mathbb{Z}$.
	\end{enumerate}
\end{theorem} 
\begin{proof}\hspace{0.3cm}
	By remark \ref{rem1}, $card(\mathcal{K}_j)$ is $\displaystyle{\frac{M}{q}}$-periodic. Then $(2)\Longleftrightarrow (3)$.
	\begin{enumerate}
		\item[]$(1)\Longrightarrow (2)$: Assume $(1)$. Let $j\in \mathbb{N}_{\frac{M}{q}}$. Then by Theorem \ref{prop3.5}, $card(\mathcal{K}_j)=rank(Z_g(j,\theta)\,)\leq qL$ since $Z_g(j,\theta)\in \mathcal{M}_{qL,p}$.
		\item[]$(2)\Longrightarrow (1)$: Assume $(2)$. By Theorem \ref{prop3.5}, it suffices to find a matrix-valued function $Z:\mathbb{N}_{\frac{M}{q}}\times [0,1[\rightarrow \mathcal{M}_{qL,p}$ such that $Z(j,.)_{r,k}\in L^2([0,1[)$ for all $(j,r,k)\in \mathbb{N}_{\frac{M}{q}}\times \mathbb{N}_{qL}\times \mathbb{N}_p$ and for all $j\in \mathbb{N}_{\frac{M}{q}}$, if $k\notin \mathcal{K}_j$, the $k$-th column of $Z(j,.)$ is identically zero, and  such that $rank(Z(j,\theta)\,)=card(\mathcal{K}_j)$. Indeed, in this case, for all $l\in \mathbb{N}_L$, define $Z_l$ as the matrix-valued function $Z_l:\mathbb{N}_{\frac{M}{q}}\times [0,1[\rightarrow \mathcal{M}_{q,p}$ defined for all $j\in \mathbb{N}_{\frac{M}{q}}$, $\theta\in[0,1[$ by $Z_l(j,\theta):=Z(j,\theta)_{lq\leq r\leq (l+1)q-1,\, 0\leq k\leq p-1}$. Then by lemma \ref{lem2.2}, there exists a unique $g_l\in \ell^2(\mathbb{S})$ such that $Z_{g_l}=Z_l$. Denote $g=\{g_l\}_{l\in \mathbb{N}_L}$, then for all $j\in\mathbb{N}_{\frac{M}{q}}, \,\theta\in [0,1[$, $Z(j,\theta)=Z_g(j,\theta)$, hence by Theorem \ref{prop3.5}, $\mathcal{G}(g,L,M,N)$ is complete in $\ell^2(\mathbb{S})$ since $rank(Z_g(j,\theta)\,)=card(\mathcal{K}_j)$. \\
		For the existence of a such matrix-valued function: Let $j\in \mathbb{N}_{\frac{M}{q}}$ and a.e $\theta \in [0,1[$.  Define a $qL\times p$ constant matrix-valued function $Z(j,.):=(Z^0(j,.),Z^1(j,.),\ldots, Z^{p-1}(j,.)\,)$ on $[0,1[$, where $Z^k(j,.)$ is the $k$-th column of $Z(j,.)$ for $k\in \mathbb{N}_p$, such that $Z^k(j,.)=0$ if $k\notin \mathcal{K}_j$ and $\{Z^k(j,.),\; k\in \mathcal{K}_j\}$ is linearly independent in $\mathbb{C}^{qL}$. This is possible since $card(\mathcal{K}_j)\leq qL$. Then for a.e $\theta \in [0,1[$,  $rank(Z(j,\theta)\,)=card(\mathcal{K}_j)$. Hence we obtain the desired matrix-valued function $Z$.
	\end{enumerate}
	
\end{proof}

\begin{proposition}\label{prop4.6}
	The following statements are equivalent:
	\begin{enumerate}
		\item There exist $ E_0,\,E_1,\, \ldots,\,E_{L-1}\subset \mathbb{Z}$ mutually disjoint such that $\mathcal{G}(\,\{\chi_{E_l}\}_{l\in \mathbb{N}_L},\, L,M,N)$ is a tight frame for $\ell^2(\mathbb{S})$ with frame bound $M$.
		\item  For all $j\in \mathbb{N}_{\frac{M}{q}}$, we have: 
		\begin{equation}
			card(\mathcal{K}_j)\leq qL.
		\end{equation} 
		\item $(11)$ holds for all $j\in \mathbb{Z}$.
	\end{enumerate}
\end{proposition}
\begin{proof}\hspace{0.3cm} By remark \ref{rem1}, $card(\mathcal{K}_j)$ is $\displaystyle{\frac{M}{q}}$-periodic. Then $(2)\Longleftrightarrow (3)$.
	\begin{enumerate}
		\item[]$(1)\Longrightarrow (2)$: By Theorem \ref{prop4.5}.
		\item[]$(2)\Longrightarrow (1)$: It suffices to find $E_0,E_1, \ldots, E_{L-1}\subset\mathbb{Z}$ mutually disjoint such that for all $l\in \mathbb{N}_L$, $E_l$ is $M\mathbb{Z}$-congruent to a subset of $\mathbb{N}_M$ and $E:=\displaystyle{\bigcup_{l\in \mathbb{N}_L}E_l}$ is $N\mathbb{Z}$-congruent to $\mathbb{S}_N$. In fact, in this case, we have  $\ell^2(\mathbb{S})=\displaystyle{\bigoplus_{n\in \mathbb{Z}}\ell^2(E+nN)}$ and, by lemma \ref{lem4.4},  $\{e^{2\pi i\frac{m}{M}.}\chi_{E_l}\}_{m\in \mathbb{N}_M,l\in \mathbb{N}_L}$ is a tight frame for $\ell^2(E)$ with frame bound $M$. Then  for all $n\in \mathbb{Z}$, $\{e^{2\pi i\frac{m}{M}.}\chi_{E_l}(.-nN)\}_{m\in \mathbb{N}_M,l\in \mathbb{N}_L}$ is a tight frame for $\ell^2(E+nN)$ with frame bbound $M$. Hence, by similar arguments used in the proof of lemma \ref{lem4.3}, $\mathcal{G}(\,\{\chi_{E_l}\}_{l\in \mathbb{N}_L}\,,L,M,N):=\{e^{2\pi i\frac{m}{M}.}\chi_{E_l}(.-nN)\}_{n\in \mathbb{Z}, m\in \mathbb{N}_M,l\in \mathbb{N}_L}$ is a tight frame for $\displaystyle{\bigoplus_{n\in \mathbb{Z}}\ell^2(E+nN)}=\ell^2(\mathbb{S})$ with frame bound $M$.\\
		{\underline{For the construction of the desired $E_l$:}} 
		Let $j\in \mathbb{N}_{\frac{M}{q}}$. Let $K$ be the maximal integer satisfying $Kq\leq card(\mathcal{K}_j)$. For all $l\in \mathbb{N}_K$, define $\mathcal{K}_j^l$ as the set of the $(l+1)$-th $q$ elements of $\mathcal{K}_j$, $\mathcal{K}_j^K$ as the set of the rest elements of $\mathcal{K}_j$ and for $l\in \mathbb{N}_L-\mathbb{N}_{K+1}$, take $\mathcal{K}_j^l=\emptyset$. For all $l\in \mathbb{N}_L$ such that $\mathcal{K}_j^l\neq \emptyset$, write $\mathcal{K}_j^l:=\{k_{l,j,i}:\; i\in \mathbb{N}_{card(\mathcal{K}_j^l)}\}$ and choose       $\{r_{l,j,i}:\; i\in \mathbb{N}_{card(\mathcal{K}_j^l)}\}\subset \mathbb{N}_q$ such that $r_{l,j,i}\neq r_{l,j,i'}$ if $i\neq i'$. This choice is guaranteed since $car(\mathcal{K}_j^l)\leq q$ for all $l\in \mathbb{N}_L$. For all $l\in \mathbb{N}_L$, define: 
		$$E_j^l=\left\lbrace
		\begin{array}{rcl}
			&\emptyset&\;\; \text{ if }\mathcal{K}_j^l=\emptyset,\\
			&\{j+k_{l,j,i}M-r_{l,j,i}N:\; i\in \mathbb{N}_{card(\mathcal{K}_j^l)}\}& \;\; \text{ otherwise}.
		\end{array}
		\right.$$
		Take for all $l\in \mathbb{N}_L$,  $E_l:=\displaystyle{\bigcup_{j\in \mathbb{N}_{\frac{M}{q}}}E_j^l}$.
		\begin{enumerate}
			\item[]$\rightarrow$ Let's show that for all $l\in \mathbb{N}_L$, $E_l$ is $M\mathbb{Z}$-congruent to a subset of $\mathbb{N}_M$. Let $l\in \mathbb{N}_L$. For this, it suffices to show that for all $j\in \mathbb{N}_{\frac{M}{q}},\, i\in \mathbb{N}_{card(\mathcal{K}_j^l)}$, we have: $$M|\, (j+k_{l,j,i}M-r_{l,j,i}N)-(j'+k_{l,j',i'}M-r_{l,j',i'}N)\;\Longrightarrow\; j=j' \text{ and }i=i'.$$
			Let $j,j'\in \mathbb{N}_{\frac{M}{q}}$ and $i,i'\in \mathbb{N}_{card(\mathcal{K}_j^l)}$ and suppose that $M|\, (j+k_{l,j,i}M-r_{l,j,i}N)-(j'+k_{l,j',i'}M-r_{l,j',i'}N)$. Then $M|\, j-j'+(k_{l,j,i}-k_{l,j',i'})M-(r_{l,j,i}-r_{l,j',i'})N$.\\
			Put $s=\displaystyle{\frac{M}{q}}$, then $M=sq$ and $N=sp$. Thus $sq|\, j-j'+(k_{l,j,i}-k_{l,j',i'})sq-(r_{l,j,i}-r_{l,j',i'})sp$, then $s|j-j'$, hence $j=j'$ since $j,j'\in \mathbb{N}_s$. On the other hand, we have $sq|(r_{l,j,i}-r_{l,j,i'})sp$, then $q|(r_{l,j,i}-r_{l,j,i'})p$, thus $q|r_{l,j,i}-r_{l,j,i'}$ since $p\wedge q=1$, hence $r_{l,j,i}=r_{l,j,i'}$ since $r_{l,j,i},r_{l,j,i'}\in \mathbb{N}_q$. And then $i=i'$.\\
			Hence for all $l\in \mathbb{N}_L$, $E_l$ is $M\mathbb{Z}$-congruent to a subset of $\mathbb{N}_M$.
			\item[]$\rightarrow$ Let's prove now that $E=\displaystyle{\bigcup_{l\in \mathbb{N}_L}E_l}$ is $N\mathbb{Z}$-congruent to $\mathbb{S}_N$. We show first that $E$ is $N\mathbb{Z}$-congruent to a subset of $\mathbb{N}_N$. For this, let $(l,j,i),\, (l',j',i')\in \mathbb{N}_L\times \mathbb{N}_{\frac{M}{q}}\times \mathbb{N}_{card(\mathcal{K}_j)}$ and suppose that $N| (j-j')+(k_{l,j,i}-k_{l',j',i'})M-(r_{l,j,i}-r_{l',j',i'})N$. Put $s=\displaystyle{\frac{M}{q}}$, then $M=sq$ and $N=sp$. Thus $sp|\, j-j'+(k_{l,j,i}-k_{l',j',i'})sq-(r_{l,j,i}-r_{l',j',i'})sp$, then $s|j-j'$, hence $j=j'$ since  $j,j'\in \mathbb{N}_s$. On the other hand, we have $sp|(k_{l,j,i}-k_{l',j,i'})sq$, then $p|k_{l,j,i}-k_{l',j,i'}$, hence $k_{l,j,i}=k_{l',j,i'}$
			since $k_{l,j,i},k_{l',j,i'}\in \mathbb{N}_p$. Then $l=l'$ and $i=i'$ by definition of the elements $k_{l,j,i}$. Thus $E$ is $N\mathbb{Z}$-congruent to a subset of $\mathbb{N}_N$. Observe that $E\subset \mathbb{S}$, then $E$ is $N\mathbb{Z}$-congruent to a subset of $\mathbb{S}_N$. By what above, we have, in particular, that the $E_j^l$ are mutually disjoint (and also the $E_l$ are mutually disjoint). Then  $$
			\begin{array}{rcl}
				card(E)&=&\displaystyle{\sum_{l\in \mathbb{N}_L}\sum_{j\in \mathbb{N}_{\frac{M}{q}}}card(\mathcal{K}_j^l)}\\
				&=&\displaystyle{\sum_{j\in \mathbb{N}_{\frac{M}{q}}} card(\mathcal{K}_j)}\\
				&=&\displaystyle{\sum_{j\in \mathbb{N}_{\frac{M}{q}}}\sum_{n\in \mathbb{Z}}\chi_{\mathbb{S}_N}(j+\displaystyle{\frac{M}{q}}n)}\;\; \text{ remark }\ref{rem1}\\
				&=&\displaystyle{\sum_{j\in \mathbb{Z}}\chi_{\mathbb{S}_N}(j)}\\
				&=&card(\mathbb{S}_N).
			\end{array}$$
			Hence $E$ is $N\mathbb{Z}$-congruent to $\mathbb{S}_N$.
		\end{enumerate}
	\end{enumerate}
\end{proof}
The following result presents an admissibility characterization for $\mathbb{S}$ to admit a  multi-window Gabor (Parseval) frame  $\mathcal{G}(g,L,M,N)$.
\begin{theorem}\label{cor4.7}
	The following statements are equivalent:
	\begin{enumerate}
		\item There exist $g:=\{g_l\}_{l\in \mathbb{N}_L}\subset \ell^2(\mathbb{S})$ such that $\mathcal{G}(g,L,M,N)$ is a Parseval frame for $\ell^2(\mathbb{S})$.
		\item There exist $g:=\{g_l\}_{l\in \mathbb{N}_L}\subset \ell^2(\mathbb{S})$ such that $\mathcal{G}(g,L,M,N)$ is a frame for $\ell^2(\mathbb{S})$.
		\item For all $j\in \mathbb{N}_{\frac{M}{q}}$ (for all $j\in \mathbb{Z}$), we have: 
		\begin{equation*}
			card(\mathcal{K}_j)\leq qL.
		\end{equation*} 
	\end{enumerate}
\end{theorem}
\begin{proof}
	We have $(1)$ impies $(2)$. And since a frame is in particular a complete sequence, then $(2)$ implies $(3)$ by Theorem $\ref{prop4.5}$. And by proposition $\ref{prop4.6}$, $(3)$ implies the existence of $\emptyset\neq E_0,E_1,\ldots,E_{L-1}\subset \mathbb{Z}$ such that $\mathcal{G}(\,\{\chi_{E_l}\}_{l\in \mathbb{N}_L},\, L,M,N)$ is a tight frame for $\ell^2(\mathbb{S})$ with frame bound $M$. Hence $\mathcal{G}(\,\{\displaystyle{\frac{1}{\sqrt{M}}}.\chi_{E_l}\}_{l\in \mathbb{N}_L},\, L,M,N)$ is a Parseval frame for $\ell^2(\mathbb{S})$.
	
\end{proof}
The following theorem presents a characterization for the admissibility of $\mathbb{S}$ to admit a $L$-window Gabor basis and $L$-window Gabor orthonormal basis $\mathcal{G}(g, L, M, N)$.
\begin{theorem}\label{prop4.8}\hspace{0.3cm} The following statements are equivalent:
	\begin{enumerate}
		\item There exist $g:=\{g_l\}_{l\in \mathbb{N}_L}\subset \ell^2(\mathbb{S})$ such that $\mathcal{G}(g,L,M,N)$ is an orthonormal basis for $\ell^2(\mathbb{S})$.
		\item There exist $g:=\{g_l\}_{l\in \mathbb{N}_L}\subset \ell^2(\mathbb{S})$ such that $\mathcal{G}(g,L,M,N)$ is a Riesz basis for $\ell^2(\mathbb{S})$.
		\item For all $j\in \mathbb{N}_{\frac{M}{q}}$ (for all $j\in \mathbb{Z}$), we have: 
		\begin{equation*}
			card(\mathcal{K}_j)= qL.
		\end{equation*} 
	\end{enumerate}
\end{theorem}
\begin{proof}
	It is well known that $(1)$ implies $(2)$. Assume that  $\mathcal{G}(g,L,M,N)$ is a Riesz basis for $\ell^2(\mathbb{S})$, then by Theorem \ref{cor4.7} we have for all $j\in \mathbb{N}_{\frac{M}{q}}$ ($\forall j\in \mathbb{Z}$) $card(\mathcal{K}_j)\leq qL$. And by proposition \ref{KRK}, we have $card(\mathbb{S}_N)=LM$. Then by lemma $\ref{lem4.1}$, we have $card(\mathcal{K}_j)= qL.$ Hence $(2)$ implies $(3)$. Assume that $card(\mathcal{K}_j)= qL.$ Then by Theorem \ref{cor4.7}, There exists $g:=\{g_l\}_{l\in \mathbb{N}_L}\subset \ell^2(\mathbb{S})$ such that $\mathcal{G}(g,L,M,N)$ is Parseval frame for $\ell^2(\mathbb{S})$. By lemma \ref{lem4.1}, we have that $card(\mathbb{S}_N)=LM$ and then by proposition \ref{KRK}, $\mathcal{G}(g,L,M,N)$ is a Riesz basis for $\ell^2(\mathbb{S})$, then is an orthonormal basis for $\ell^2(\mathbb{S})$ (lemma \ref{PRO}). Hence $(3)$ implies $(1)$.
\end{proof}
\begin{remark}
	In the case of $\mathbb{S}=\mathbb{Z}$, we hace for all $j\in \mathbb{N}_{\frac{M}{q}}$, $\mathcal{K}_j=\mathbb{N}_p$. Then the condition $(3)$ in the Theorem \ref{cor4.7} is equivalent to $p\leq Lq$ which is equivalent to $N\leq LM$. Then we obtain the proposition $3.5$ in \cite{3}. And also the condition $(3)$ in Theorem \ref{prop4.8} is equivalent to $N=LM$. Then we obtain the proposition 3.11 in \cite{3}.
\end{remark}
We finish this work by the following example:
\begin{example}
	In this example, we use the notations already introduced in what above.\\
	Let $M=3$ and $N=5$. Let $\mathbb{S}=\{0,1,2,4\}+5\mathbb{Z}$. It is clear that $p=5$ and $q=3$. Then $\displaystyle{\frac{M}{q}}=1$, then $\mathbb{N}_{\frac{M}{q}}=\{0\}$. We have clearly $\mathcal{K}_0=\{0,2,3,4\}$. Then $card(\mathcal{K}_0)=4>q$. Then, by Theorem \ref{cor4.7},  there does not exist a Gabor frame with a signe window for $\ell^2(\mathbb{S})$, but by the same Theorem, we can always find a Multiwindow Gabor frame for $\ell^2(\mathbb{S})$ with $L$-window for all $L\geqslant 2$ since $card(\mathcal{K}_j)=4\leq3\times2=6$. Here is an example of $2$-window Gabor frame for $\ell^2(\mathbb{S})$.\\
	Define $g_0:=\chi_{\{-1,0,1\}}$ and $g_1:=\chi_{\{-4,4,12\}}$, since $-1,0,1,-4,4,12\in \mathbb{S}$, then $g_0,g_1\in \ell^2(\mathbb{S})$. Observe, also, that $\mathbb{S}=\{0,1,2,4,5,-,7,9,10,11,12,14\}+15\mathbb{Z}$. Then we have $g_0$ vanishes on $\{-10, -5, -4, 2, 3, 4, 6, 7, 8, 9, 12, 13\}+15\mathbb{Z}\bigcup \{-1,0,1\}+15(\mathbb{Z}-\{0\})$, and $g_1$ vanishes on $\{-10, -5, -1, 0, 1, 2, 3, 6, 7, 8, 9, 13\}+15\mathbb{Z}\bigcup \{-4,4,12\}+15(\mathbb{Z}-\{0\})$. Then, after a simple computation, we have for a.e $\theta \in [0,1[$: 
	$$\begin{array}{rcl}
		Z_{g_0}(0,\theta)=\begin{pmatrix}
			1 & 0 & 0 & 0 & 0 \\
			0 & 0 & 1 & 0 & 0 \\
			0 & 0 & 0 & 1 & 0
		\end{pmatrix}, 
		Z_{g_1}(0,\theta)=\begin{pmatrix}
			0 & 0 & 0 & 0 & 1 \\
			0 & 0 & 0 & 1 & 0 \\
			0 & 0 & 1 & 0 & 0
		\end{pmatrix}, 
	\end{array} $$
	Then for all $x:=(x_0,x_1,\ldots,x_4)\in \mathbb{C}^5$, we have: $\langle Z_{g_0}(0,\theta)^*Z_{g_0}(0,\theta)x,x\rangle =\vert x_0\vert^2+\vert x_2\vert^2+\vert x_3\vert^2$ and  $\langle Z_{g_1}(0,\theta)^*Z_{g_1}(0,\theta)x,x\rangle =\vert x_2\vert^2+\vert x_3\vert^2+\vert x_4\vert^2.$ Then $\langle Z_{g_0}(0,\theta)^*Z_{g_0}(0,\theta)x,x\rangle +\langle Z_{g_1}(0,\theta)^*Z_{g_1}(0,\theta)x,x\rangle=\vert x_0\vert^2+2\vert x_2\vert^2+2\vert x_3\vert^2+\vert x_4\vert^2.$  Since $\langle \mathcal{K}(0)x,x\rangle=\vert x_0\vert^2+\vert x_2\vert^2+\vert x_3\vert^2+\vert x_4\vert^2$, then we obtain: 
	$$\langle \mathcal{K}(0)x,x\rangle \leq \langle Z_{g_0}(0,\theta)^*Z_{g_0}(0,\theta)x,x\rangle +\langle Z_{g_1}(0,\theta)^*Z_{g_1}(0,\theta)x,x\rangle \leq 2\langle \mathcal{K}(0)x,x\rangle.$$
	Hence, by Theorem \ref{prop3.6}, $\mathcal{G}(\{g_0,g_1\}, 2,3,5)$ is a $2$-window Gabor frame for $\ell^2(\mathbb{S})$ with frame bounds $3$ and $6$.
\end{example}

\medskip

	\section*{Acknowledgments}
	It is my great pleasure to thank the referee for his careful reading of the paper and for several helpful suggestions.
	
	\section*{Ethics declarations}
	
	\subsection*{Availablity of data and materials}
	Not applicable.
	\subsection*{Conflict of interest}
	The author declares that he has no competing interests.
	\subsection*{Fundings}
	Not applicable.

\end{document}